\documentclass[12pt]{amsart}

\RequirePackage[colorlinks,citecolor=blue,urlcolor=blue]{hyperref}
\allowdisplaybreaks[4]
\DeclareMathAlphabet{\mathcal}{OMS}{cmsy}{m}{n}

\usepackage{epsfig}
\usepackage{graphics}
\usepackage{amsfonts}
\usepackage{amsmath}
\usepackage{amscd,amssymb,mathrsfs}
\usepackage{latexsym}

\newtheorem{theorem}{Theorem}[section]
\newtheorem{lemma}{Lemma}[section]
\newtheorem{proposition}{Proposition}[section]

\newtheorem{remark}{Remark}[section]

\begin{document}

\title[The Graph and Range Singularity Spectra of Random wavelet series]{The Graph and Range Singularity Spectra of Random Wavelet Series built from Gibbs measures}

\author{Xiong Jin}
\address{INRIA Rocquencourt, B.P. 105, 78153 Le Chesnay Cedex, France}
\email{xiongjin82@gmail.com}

\begin{abstract}
We consider multifractal random wavelet series built from Gibbs measures, and study the singularity spectra associated with the graph and range of these functions restricted to their iso-H\"older sets. To obtain these singularity spectra, we use a family of Gibbs measures defined on a sequence of topologically transitive subshift of finite type whose Hausdorff distance to the set of zeros of the mother wavelet tends to $0$.

\end{abstract}

\maketitle

\section{Introduction}

\subsection{The graph and range singularity spectra}

Given a function $f:[0,1]\mapsto \mathbb{R}$ and a subset $E\subset [0,1]$, the sets on the graph and range of $f$ over $E$ are defined by
\[
G_f(E)=\{(x,f(x)):x\in E\} \text{ and } R_f(E)=\{f(x):x\in E\}.
\]
It is classical in probability and geometric measure theories to study the Hausdorff dimension of these sets for non smooth functions. The first works on these questions can be traced back to L\'evy~\cite{Levy} and Taylor~\cite{Tay}, regarding the Hausdorff dimension and the Hausdorff measure of the range of Brownian motion. Since then, many progresses have been made for fractional Brownian motions, stable L\'evy processes and many other processes and functions, see~\cite{BlGe,Horo,Pru,Be,Ka,MaWi,PU,Urb,BeUr,HuLau,Hunt,Rou,Xiao2,Fe,DeFa} and others. In these contexts, the H\"older regularity of the function plays an important role. A nature way to describe such a regularity is the pointwise H\"older exponent:
\[
h_f(x)=\liminf_{r\to 0^+} \frac{1}{\log r} \log \left(\sup_{s,t\in B(x,r)} \big|f(s)-f(t)\big|\right).
\]
It is known that the minimal value of $h_f$ over the set $E$ provides upper bounds of the Hausdorff dimension of $G_f(E)$ and $R_f(E)$. In~\cite{Jin}, as a generalization of Lemma 8.2.1 in~\cite{Adl}, Theorem 6 of Chapter 10 in~\cite{Ka} and Lemma 2.2 in~\cite{Xiao1}, one has the following result:

\medskip

\noindent
{\bf{Theorem A}}~\cite{Jin} {\it Let $\dim_H$ be the Hausdorff dimension. Suppose that $\inf_{x\in E} h_f(x)=h>0$. Then
\begin{eqnarray*}
\dim_H G_f(E) &\le& \left(\frac{\dim_H E}{h} \wedge \left(\dim_H E+1-h\right)\right) \vee \dim_H E, \\
\dim_H R_f(E) &\le& \frac{\dim_H E}{h}\wedge 1.
\end{eqnarray*}
}

When $f$ is monofractal, like fractional Brownian motion or Weierstrass function, the exponent $h_f$ is a constant function independent of the set $E$. But for most of non smooth functions (see~\cite{BuZo,JAFFJMPA,FrJa} for instance), the behavior of $h_f$ is actually very irregular: it varies wildly from one point to another. To describe such a behavior, physicists~\cite{hentschel,FrischParisi,HaJeKaPrSh} introduced the so-called multifractal analysis, which consists in computing the Hausdorff dimension of the iso-H\"older sets:
\[
E_f(h)=\big\{x\in [0,1]: h_f(x)=h\big\}, \ h\ge 0.
\]
The function
\[
d_f:h\ge 0 \mapsto \dim_H E_f(h)
\]
is called the singularity spectrum of $f$, and $f$ is said to be multifractal if $E_f(h)\neq\emptyset$ for at least two distinct values of $h$.

The singularity spectrum of a function describes the distribution of its H\"older exponents from the macroscopic point of view: it tells how large is the set of points at which the function has a given H\"older exponent. This spectrum has been computed for certain classes of functions, as well as for some classical functions,  including  Riemann's nowhere differentiable function, L\'evy processes, L\'evy processes in multifractal time, self-similar functions, wavelet series or generic functions in certain Besov or Sobolev spaces, as well as indefinite integrals of positive measures \cite{Rand89,BMP,HoWa,JaRMI,Pesin,Ja97,Jafsiam1,JAFFJMP,JaPTRF,BuZo,JAFFJMPA,JA2002,Jafw,BSw,BSlevy,FrJa,Se,BJ}.

Inspired by the important role that H\"older exponents played in both multifractal analysis and dimension problems on the graph and range, it is interesting to find the following singularity spectra:
\[
d^{\, S}_f: h\ge 0 \mapsto \dim_H S_f(h),\ S\in\{G,R\},
\]
where for $h\ge 0$ we note
\[
G_f(h)=G_f(E_f(h)) \text{ and } R_f(h)=R_f(E_f(h)).
\]

In~\cite{Jin} we studied these singularity spectra for a class of random multifractal functions, namely the $b$-adic independent cascade function, which can be viewed as a generalization of the Mandelbrot cascades introduced in \cite{Mand}. In this paper this study is done for another class of random multifractal functions constructed in~\cite{BSw}: the random wavelet series built from Gibbs measures.  Before going to the details, let us first give some backgrounds and notations on the wavelet series and multifractal analysis.

\subsection{Orthogonal wavelet basis and multifractal analysis}

Let $\psi$ be an $r_0$-smooth mother wavelet on $\mathbb{R}$, with $r_0\in \mathbb{N}^*$, so that the functions $\{\psi_{j,k}=\psi(2^{j}\cdot-k)\}_{(j,k)\in\mathbb{Z}^2}$ form an orthogonal wavelet basis of $L^2(\mathbb{R})$ (see \cite{Me} for instance for the definition and construction). Each function $f\in L^2(\mathbb{R})$ can be written as
\[
f(x)=\sum_{(j,k)\in\mathbb{Z}^2} d_{j,k}\cdot \psi_{j,k}(x),\ \text{where }
\]
where the wavelet coefficient $d_{j,k}$ is given by
\[
d_{j,k}=2^j\int_{\mathbb{R}} f(t) \cdot \psi_{j,k}(t)\, \mathrm{d} t.
\]

It is known that the asymptotic behavior of the wavelet coefficients provides fine informations on the H\"older regularity of the function. For example, due to Proposition 4 in~\cite{Jafw}, if there exist constants $C_0>0$, $\epsilon_0\in (0,1)$ such that $|d_{j,k}|\le C_02^{-\epsilon_0j}$ for each $j\ge 0$ and $k\in \mathbb{Z}$, then $f$ is $\epsilon_0$-H\"older continuous, that is there exist $C>0$, $\delta>0$ such that for any $x,y\in\mathbb{R}$ with $|x-y|\le \delta$, we have $|f(x)-f(y)|\le C|x-y|^{\epsilon_0}$. Moreover, once $f$ is $\epsilon_0$-H\"older continuous, one can also obtain the pointwise regularity of $f$ from its wavelet coefficients: for each $(j,k)\in\mathbb{Z}^2$ define the wavelet leader
\[
L_{j,k}=\sup \left\{ |d_{j',k'}| : (j',k')\in\mathbb{Z}^2,\ [k'2^{-j'},(k'+1)2^{-j'})\subset [k2^{-j},(k+1)2^{-j})\right\}
\]
and for $x_0\in\mathbb{R}$ and $j\ge 0$ define the coefficient
\[
L_j(x_0)=\sup \left\{ L_{j,k} : k\in \mathbb{Z},\ x_0\in [(k-1)2^{-j},(k+2)2^{-j}) \right\}
\]
and the exponent 
\begin{equation}\label{HE}
\bar{h}_f(x_0)=\liminf_{j\to +\infty} -j^{-1} \log_2 L_j(x_0).
\end{equation}
Then, due to Corollary 1 in~\cite{Jafw},  for any $x_0\in\mathbb{R}$, if $[\bar{h}_f(x_0)]\le r_0$ we have   
\[
\bar{h}_f(x_0)=\sup \big\{h>0:\exists P\in \mathbb{C}[x], |f(x)-P(x-x_0)|=O(|x-x_0|^h), x\to x_0\big\},
\]
that is $\bar{h}_f(x_0)$ provides another very natural pointwise exponent for $f$ at $x_0$, whose connection with $h_f(x_0)$ is explained in the following remark.
\begin{remark}
{\rm By definition we have $h_f(x_0)\le \bar{h}_f(x_0)$, and $h_f(x_0)=\bar{h}_f(x_0)$ if neither of them is an integer. The difference between these two exponents is that $\bar{h}_f(x_0)$ is not influenced by addition of a polynomial function, whereas $h_f(x_0)$ describes directly the oscillation of function $f$ around $x_0$, and his sensible to the addition of a polynomial function.
}\end{remark}

Wavelet expansion is thus an effective tool to study the local regularity of a function. It is also connected to the Hausdorff spectrum as follows. Define the scaling function of $f$ as
\begin{equation}\label{xi}
\xi_{f}:q\in \mathbb{R}\mapsto \xi_{f}(q)=\liminf_{j\to+\infty}-j^{-1}\log_2\sum_{k\in\mathbb{Z}: [k2^{-j},(k+1)2^{-j}]\subset [0,1], L_{j,k}\neq 0} |L_{j,k}|^q.
\end{equation}
Then, if $r_0$ is large enough so that $\bar{h}_f(x)\le r_0$ for all $x\in [0,1]$ ,we have
\begin{equation}\label{xiup}
\dim_H \{x\in (0,1): \bar{h}_f(x)=h\} \le \xi_f^*(h):=\inf_{q\in\mathbb{R}} q\cdot h-\xi_f(q),\ h\ge 0,
\end{equation}
where a negative dimension means that the set is empty~\cite{Jabb,Jafw}. One says that the restriction of $f$ to $[0,1]$ fulfills the multifractal formalism at $h\ge 0$ if the above inequality is an equality.

\subsection{Random wavelet series built from multifractal measure}\label{1.3}

In~\cite{BSw}, Barral and Seuret construct a class of wavelet series by directly taking the wavelet coefficients built from some well-known multifractal measures,  in such a way that the Hausdorff spectrum of the wavelet series can be directly deduced from that of the measure. 

Specifically, let $\mu$ be a positive Borel measure on $\mathbb{R}$ supported by the interval $[0,1]$, and define
\[
F_\mu(x)=\sum_{j=0}^{\infty} \sum_{k=0}^{2^{j}-1} d_{j,k}\cdot \psi_{j,k}(x), \text{ where }d_{j,k} =\pm 2^{-j(s_0-1/p_0)} \mu([k2^{-j},(k+1)2^{-j}))^{1/p_0},
\]
with $s_0,p_0> 0$ and $s_0-1/p_0> 0$.

Denote by $I_{j,k}=[k2^{-j},(k+1)2^{-j})$. Notice that in this setting, the wavelet leader $L_{j,k}$ is nothing but $|d_{j,k}|=2^{-j(s_0-1/p_0)} \mu(I_{j,k})^{1/p_0}$, thus due to~\eqref{xi}, we have
\begin{equation}\label{xisp}
\xi_{F_\mu}(q)=q(s_0-1/p_0)+\tau_\mu(q/p_0), \ q\in\mathbb{R},
\end{equation}
where the so called R\'enyi entropy or $L^q$ spectrum of $\mu$ is 
\begin{equation}\label{tau}
\tau_{\mu}(q)=\liminf_{n\to\infty}-j^{-1}\log_2 \sum_{k=0,\cdots,2^{j-1}} \mathbf{1}_{\left\{\mu(I_{j,k})\neq 0\right\}}\cdot \mu(I_{j,k})^q,\ q\in\mathbb{R}.
\end{equation}
By construction, $F_\mu$  is $(s_0-1/p_0)$-H\"older continuous and belongs to the Besov space $B^{s_0,\infty}_{p_0}(\mathbb{R})$ if $s_0<r_0$. Moreover, if $\mu$ fulfills the multifractal formalism for measures at $\alpha\ge 0$ (in the sense of ~\cite{BMP}), then the restriction of $F_\mu$ to $[0,1]$ fulfills the multifractal formalism described above at $h=s_0-1/p_0+\alpha/p_0$ when $[h]\le r_0$ (see~\cite{BSw}).

From now on, $F_\mu$ stands for the restriction of $F_\mu$ to $[0,1]$.

In~\cite{BSw}, Barral and Seuret also considered some random multiplicative perturbation of $F_\mu$. It consists in considering a sequence of independent random variables $\{\pi_{j,k}\}_{j\ge 0,\ k\in\{0,1,\cdots,2^{j}-1\}}$ and then the wavelet series $F_\mu^{\text{pert}}$ on $[0,1]$ whose coefficients  are given by $d^{\text{pert}}_{j,k}= \pi_{j,k}\cdot d_{j,k}$:
\[
F^{\text{pert}}_\mu(x)=\sum_{j=0}^{\infty} \sum_{k=0}^{2^{j}-1} \pi_{j,k}\cdot (\pm 2^{-j(s_0-1/p_0)} \mu(I_{j,k})^{1/p_0} )\cdot \psi_{j,k}(x).
\]
 Under certain conditions on the moments of $\pi_{j,k}$, for example,
\begin{itemize}
\item[(A1)] For any $q\in\mathbb{R}$ we have $\sup_{j\ge 0} \sup_{k=0,1,\cdots,2^{j}-1} \mathbb{E}(|\pi_{j,k}|^q)<\infty$.
\end{itemize}
they show that, with probability 1, $F_\mu^{\text{pert}}$ fulfills the multifractal formalism at $h$ whenever $F_\mu$ does.

\subsection{Main result}\label{1.4}

This paper studies the graph and range singularity spectra for random wavelet series $F_\mu^{\text{pert}}$, where $\mu$ is the canonical image on $[0,1]$ of a Gibbs measure $\mu_\varphi$ associated with a H\"older potential $\varphi$ on a symbolic space $\Sigma$ (see Section~\ref{SFTGibbs} and \ref{randomspec} for precise definitions). 

The random perturbation of $F_\mu$ is essential to our approach based on the potential theoretic method for the estimation of Hausdorff dimensions (see Chapter 4 in~\cite{Falc}). The efficiency of the combination between randomness and potential theoretic method has been used to compute the Hausdorff dimension of the whole graph of classical processes \cite{Falc}, random Weierstrass function \cite{Hunt} and random wavelet series \cite{Rou} (see also \cite{JAFFJMP,Rou1} for questions related to the dimensions of the whole graph of wavelet series). 

\medskip

In addition to (A1), we assume:
\begin{itemize}
\item[(A2)] The mother wavelet $\psi$ has only finite many zeros on $[0,1]$.
\item[(A3)] Each random variable $\pi_{j,k}$ has a bounded density function $f_{j,k}$ and for any $\epsilon>0$ we have $\sum_{j\ge 0} (\sup_{k=0,\cdots, 2^{j}-1}\|f_{j,k}\|_\infty) \cdot 2^{-j\epsilon}<\infty$.
\end{itemize}
Under these assumptions we prove the following result:
\begin{theorem}\label{mainthm}
With probability 1 for all $h\in(0,1)$ such that $\xi_{F_\mu}^*(h)>0$,
\begin{eqnarray*}
d^{\, G}_{F_\mu^{\text{pert}}}(h) &=& \frac{d_{F_\mu^{\text{pert}}}(h)}{h} \wedge \Big(d_{F_\mu^{\text{pert}}}(h)+1-h\Big)=\frac{\xi_{F_\mu}^*(h)}{h} \wedge \Big(\xi_{F_\mu}^*(h)+1-h\Big), \\
d^{\, R}_{F_\mu^{\text{pert}}}(h) &=& \frac{d_{F_\mu^{\text{pert}}}(h)}{h} \wedge 1= \frac{\xi_{F_\mu}^*(h)}{h} \wedge 1.
\end{eqnarray*}
\end{theorem}

\begin{remark}{\rm 
(1) Since we are dealing with sets on the graph and range, it is more convenient to use the oscillating exponent $h_{F_\mu^{\text{pert}}}$. But while transferring the local dimension of Gibbs measure $\mu$ to the H\"older exponent of wavelet series $F_\mu^{\text{pert}}$, we have to use $\bar{h}_{F_\mu^{\text{pert}}}$. So to avoid complications,  we only consider the iso-H\"older set $E_{F_\mu^{\text{pert}}}(h)$ for $h\in (0,1)$, since in this case $h_{F_\mu^{\text{pert}}}$ and $\bar{h}_{F_\mu^{\text{pert}}}$ are equal everywhere on the set $E_{F_\mu^{\text{pert}}}(h)$. For $h\ge 1$, it is clear that $\dim_H G_{F_\mu^{\text{pert}}}(h)=\dim_H E_{F_\mu^{\text{pert}}}(h)$, which provides us with the whole graph spectrum. But we have no result for the range singularity spectrum for $h\ge 1$.

(2) Notice that our result is uniform. It is valid almost surely for all $h\in (0,1)$ with $\xi_{F_\mu}^*(h)>0$, and not just for each $h\in (0,1)$ almost surely.

}\end{remark}

Let us roughly explain our strategy to prove Theorem~\ref{mainthm}. We apply the potential theoretic method to families of images of Gibbs measures on the graph and range of $F_\mu^{\text{pert}}$. We must consider the restrictions of the potentials $(q\varphi)_{q\in\mathbb{R}}$ on a sequence $\{X_{k}\}_{k\ge 1}$ of subshifts of finite type of $\Sigma$ whose canonical projection $\widetilde X_k$ in $[0,1]$ has a positive Hausdorff distance to the set of zeros of $\psi$, which tends to $0$ as $k$ tends to $\infty$. We also need to consider  the canonical  projections on $[0,1]$ of the equilibrium states of these restricted potentials, that we denote by $\{(\mu_q^{(k)})_{q\in\mathbb{R}}\}_{k\ge 1}$. Then for each $k\ge 1$, there exists an interval $J_k$ such that for each $q\in J_k$,  there exists an exponent $h^{(k)}_q\in (0,1)$  as well as $E_q^{(k)}\subset E_{F_\mu^{\text{pert}}}(h^{(k)}_q)\cap \widetilde X_k$ such that $\mu_q^{(k)}(E_q^{(k)})>0$, and two numbers $\gamma_{q,G}^{(k)},\gamma_{q,R}^{(k)}>0$ such that for any $\delta>0$ small enough, almost surely, for all $q\in J_k$
\[
\iint_{s,t\in E_q^{(k)}} (|F_\mu^{\text{pert}}(s)-F_\mu^{\text{pert}}(t)|^2+|s-t|^2)^{-(\gamma_{q,G}^{(k)}-\delta)/2} \mathrm{d} \mu_q^{(k)}(s)\mathrm{d} \mu_q^{(k)}(t) <\infty
\]
and
\[
\iint_{s,t\in E_q^{(k)}} |F_\mu^{\text{pert}}(s)-F_\mu^{\text{pert}}(t)|^{-\gamma_{q,R}^{(k)}+\delta} \mathrm{d} \mu_q^{(k)}(s)\mathrm{d} \mu_q^{(k)}(t) <\infty.
\]
This yields the almost sharp lower bounds
\[d^{\, G}_{F_\mu^{\text{pert}}}(h^{(k)}_q)\ge \gamma_{q,G}^{(k)}-\delta \text{ and } d^{\, R}_{F_\mu^{\text{pert}}}(h^{(k)}_q)\ge \gamma_{q,R}^{(k)}-\delta,\]
and by letting $k$ tend to $\infty$ we can get the sharp lower bound in Theorem~\ref{mainthm} (See Section~\ref{Lbe} for details).

The reason why we must consider subshift of finite type that avoids zeros of $\psi$ is that for  proving the finiteness of the above integrals, we have to control from below the increment $|\psi(s)-\psi(t)|$ when $s\in E^{(k)}_q$ and $t$ is far away from $s$. This is possible only if $s$ is never too close to the zeros of $\psi$. 

\smallskip

Here we must mention the results of Roueff in~\cite{Rou} which deals with the Hausdorff dimension of whole graph of random wavelet series. Briefly speaking, let $\{c_{j,k}\}_{j\ge 0, k=0,\cdots,2^j-1}$ be a sequence of real valued random variables whose laws are absolutely continuous with respect to Lebesgue measure. Let $\mathcal{T}(c_{j,k})$ stand for the $L^\infty$ norm of the density of $c_{j,k}$. Roueff proves that (Theorem 1 in~\cite{Rou}) if $\psi$ has finitely many zeros on $[0,1]$, then the Hausdorff dimension of the graph of the random wavelet series
\[
F(x)=\sum_{j\ge 0}\sum_{k=0}^{2^j-1} c_{j,k}\cdot \psi_{j,k}(x)
\]
is almost surely larger than or equal to
\[
\limsup_{J\to\infty}\liminf_{j\to\infty}\frac{\log \min_{i=j}^{j+J}\left\{\sum_{k=0}^{2^i-1}\min\{1,\mathcal{T}(c_{i,k})\cdot 2^{-i}\}\cdot \nu(I_{i,k})^2\right\}}{-j\log2},
\]
where $\nu$ can be chosen as any probability measure on $[0,1]$ such that there exists a constant $C$ and $s>0$ such that for any Borel sets $A\subset [0,1]$ and $B\subset A$ such that $\nu (A)>0$ we have $\nu(B)/\nu(A)\le C (|B|/|A|)^s$, where $|B|$, $|A|$ stand for the diameters of $A$ and $B$. Due to the scaling properties of the equilibrium state $\mu_{q\varphi}$ of each potential $q\varphi$, $q\in\mathbb{R}$, it is natural to try using Roueff's approach to our problem: for $j\ge 0$ and $k=0,\cdots,2^j-1$ we take
\[
c_{j,k}=\pi_{j,k}\cdot d_{j,k}, \text{ where } d_{j,k} =\pm 2^{-j(s_0-1/p_0)} \mu(I_{j,k})^{1/p_0},
\]
this, together with (A3) and the definition of $\mathcal{T}(c_{j,k})$, gives us
\[
\mathcal{T}(c_{j,k})=\|f_{j,k}\|_{\infty}\cdot |d_{j,k}|^{-1}
\]
Then, for $q\in\mathbb{R}$ and $\epsilon>0$ define 
\[
E_n(q,\epsilon)=\left\{x\in[0,1]: \forall\ j\ge n, \ \left\{\begin{array}{l}
|d_{j,k_{j,x}}|\in[2^{-j(h_q+\epsilon)},2^{-j(h_q-\epsilon)}],\\
\mu_{q}(I_{j,k_{j,x}})\in [2^{-j(\xi_{F_\mu}^*(h_q)+\epsilon)}, 2^{-j(\xi_{F_\mu}^*(h_q)-\epsilon)}]
\end{array} \right. \right\},
\]
where $\mu_q$ is the canonical projection of $\mu_{q\varphi}$ on $[0,1]$, $k_{j,x}$ is the unique integer $k$ such that $x\in I_{j,k}=[k2^{-j},(k+1)2^{-j})$ and $h_q=s_0-1/p_0+\tau'_\mu(q)/p_0$. Due to Section~\ref{magibbs} and~\ref{randomspec} we have for any $\epsilon>0$, $\mu_{q\varphi}(\lim_{n\to\infty} E_n(q,\epsilon))=1$. By continuity we can find an integer $N_{q,\epsilon}$ such that $\mu_{q\varphi}(E_{N_{q,\epsilon}}(q,\epsilon))>0$. Now we take $\nu_q=\mu_{q\varphi}\big|_{E_{N_{q,\epsilon}}(q,\epsilon)}$ and then define $\nu$ by $\nu(B)=\nu (B\cap E_{N_{q,\epsilon}})$ for each Borel subset of $[0,1]$ (here $\nu(B)/\nu(A)\le C (|B|/|A|)^s$ holds with some constant $C$, $s=\xi_{F_\mu}^*(h_q)-\epsilon$, and $A$ and $B$ dyadic intervals). Then for $j>N_{q,\epsilon}$, we have
\begin{eqnarray*}
&&\sum_{k=0}^{2^j-1}\min\{1,\mathcal{T}(c_{j,k})\cdot 2^{-j}\}\cdot \nu(I_{j,k})^2\\
&\le& \sum_{k=0}^{2^j-1}\min\{1,\|f_{j,k}\|_\infty\cdot |d_{j,k}|^{-1}\cdot 2^{-j}\} \cdot \mu_{q\varphi}(I_{j,k} \cap E_{N_{q,\epsilon}}(q,\epsilon) )^2 \\
&\le& \sum_{k=0}^{2^j-1}\min\{1,\|f_{j,k}\|_\infty\cdot 2^{-j(1-h_q-\epsilon)}\} \cdot 2^{-2j(\xi_{F_\mu}^*(h_q)-\epsilon)}\cdot \mathbf{1}_{\left\{\mu_{q\varphi}(I_{j,k}) \ge 2^{-j(\xi_{F_\mu}^*(h_q)+\epsilon)}\right\}}\\
&\le&\sum_{k=0}^{2^j-1}\min\{1,\|f_{j,k}\|_\infty\cdot 2^{-j(1-h_q-\epsilon)}\} \cdot 2^{-2j(\xi_{F_\mu}^*(h_q)-\epsilon)}\cdot 2^{j(\xi_{F_\mu}^*(h_q)+\epsilon)}\cdot \mu_{q\varphi}(I_{j,k}) \\
&\le& \min\{2^{j(1-h_q-\epsilon)},\sup_{k=0,\cdots,2^{j}-1}\|f_{j,k}\|_\infty\}\cdot 2^{-j(\xi_{F_\mu}^*(h_q)+1-h_q-4\epsilon)} \sum_{k=0}^{2^{j}-1} \mu_{q\varphi}(I_{j,k}) \\
&\le&\big(\sup_{k=0,\cdots,2^{j}-1}\|f_{j,k}\|_\infty\big) \cdot 2^{-j(\xi_{F_\mu}^*(h_q)+1-h_q-4\epsilon)} 
\end{eqnarray*}
Then under assumptions (A1-3), due to the fact that $\mu_{q}$ is carried by the set $E_{F_\mu^{\text{pert}}}(h_q)$ almost surely, Roueff's result implies that, if $\xi_{F_\mu}^*(h_q)+1-h_q-4\epsilon>1$ (this is essentially in his proof), then almost surely
\[
d^{\, G}_{F_\mu^{\text{pert}}}(h_q)\ge \xi_{F_\mu}^*(h_q)+1-h_q-4\epsilon.
\]
By taking a sequence of $\epsilon$ tending to $0$, we get the sharp lower bound for $d^{\, G}_{F_\mu^{\text{pert}}}(h_q)$ given by Theorem~\ref{mainthm} when $d^{\, G}_{F_\mu^{\text{pert}}}(h_q)>1$. But this result holds only ``for each $q\in\mathbb{R}$ almost surely", so is not uniform like Theorem~\ref{mainthm}, and it seems that Roueff's method cannot yield such a result, nor the value of $d^{\, G}_{F_\mu^{\text{pert}}}(h_q)$ when $d^{\, G}_{F_\mu^{\text{pert}}}(h_q)\le 1$ and $h_q\le 1$. 

\medskip

The rest of the paper is organized as follows: Section~\ref{SFTGibbs} gives some definitions and notations about subshift of finite type, Gibbs measure and its multifractal analysis. Sections~\ref{prfmainthm} and ~\ref{prfkeythm} provide the proof of Theorem~\ref{mainthm}.

\section{Subshift of finite types, Gibbs measures and multifractal analysis}\label{SFTGibbs}

\subsection{Subshift of finite type} 
Let $\Sigma=\{0,1\}^\mathbb{N}$ and $\Sigma_*=\bigcup_{n\ge 0} \Sigma_n$, where $\Sigma_0=\{\varnothing\}$ and $\Sigma_n=\{0,1\}^n$ for $n\ge 1$.

Denote the length of $w$ by $|w|=n$ if $w\in\Sigma_n$, $n\ge 0$ and $|w|=\infty$ if $w\in\Sigma$.

For $w\in\Sigma_*$ and $t\in\Sigma_*\bigcup \Sigma$, the concatenation of $w$ and $t$ is denoted by $w\cdot t$ or $wt$.

For $w\in \Sigma_*$, the cylinder with root $w$, i.e. $\{w\cdot u: u \in \Sigma\}$ is denoted by $[w]$.

The set $\Sigma$ is endowed with the standard metric distance
\[
\rho(s,t)=\inf\{2^{-n}: n\ge 0,\  \exists \ w\in\Sigma_n \text{ such that } s,t\in [w]\}.
\]
Then $(\Sigma,\rho)$ is a compact metric space. Denote by $\mathcal{B}$ the Borel $\sigma$-algebra with respect to $\rho$. Clearly $\mathcal{B}$ can be generated by the cylinders $[w]$, $w\in\Sigma_*$.

If $n\ge 1$ and $w=w_1\cdots w_n\in \Sigma_n$ then for every $0\le i\le n$, we write $w|_i=w_1\dots w_i$, with the convention $w|_0=\varnothing$. Also, for any infinite word $t=t_1 t_2\cdots \in\Sigma$ and $i\ge 0$, we write $t|_i=t_1\dots t_i$, with the convention $t|_0=\varnothing$.

For $t\in\Sigma$ define the left side shift $\sigma: \Sigma\mapsto\Sigma$ by
\[\sigma(t_1t_2\cdots)=t_2t_3\cdots.\]

A subshift is a $\sigma$-invariant compact set $X\subset \Sigma$, that is $\sigma(X)\subset X$.

A subshift $X$ is said to be of finite type if there is an admissible set $A\subset \Sigma_n$ for some $n\ge 2$ such that
\[
X=\{t\in\Sigma: \sigma^m(t)|_n\in A, \ \forall\ m\ge 0 \}.
\]
The admissible set $A$ induces a transition matrix $B: \Sigma_{n-1}\times\Sigma_{n-1} \mapsto \{0,1\}$ with $B(a_1\cdots a_{n-1}, a_2\cdots a_n)=1$ if $a_1\cdots a_n \in A$, and $B(i,j)=0$ otherwise. Then $X$ can be redefined as
\[
X=\{t \in \Sigma: B(\sigma^m(t)|_{n-1}, \sigma^{m+1}(t)|_{n-1})=1,\ \forall\ m\ge 0\}.
\]

The dynamical system $(X,\sigma)$ is called topologically transitive (resp. mixing) if $B$ is irreducible, that is for any $i,j\in \Sigma_{n-1}$ there is a $k\ge 1$ such that $B^k(i,j)>0$ (resp. if $B$ is primitive, that is there is a $k\ge 1$ such that $B^k(i,j)>0$ for all $i,j\in\Sigma_{n-1}$).

\subsection{Gibbs measure on topologically transitive subshift of finite type}

Let $\varphi$ be a H\"older continuous function defined on $\Sigma$, which will be mentioned as a H\"older potential in the following.

Let $(X,\sigma)$ be a topologically transitive subshift of finite type of the full shift $(\Sigma,\sigma)$.

For $n\ge 1$ the $n^{\text{th}}$-order Birkhoff sum of $\varphi$ over $\sigma$ is the function
\[
S_n\varphi(t)=\sum_{i=0}^{n-1} \varphi\circ \sigma^i(t),\ t\in\Sigma.
\]

The topological pressure of $\varphi$ on $X$ is defined by
\begin{equation}\label{pressure}
P_X(\varphi)=\lim_{n\to \infty} \frac{1}{n}\  \log \sum_{w\in\Sigma_n: [w]\cap X\neq \emptyset} \exp\left(\max_{t\in [w]} S_n\varphi (t)\right)
\end{equation}
(the existence of the limit is ensured by sub-additivity properties of the logarithm on the right hand side).

It follows from the thermodynamic formalism developed by Sinai, Ruelle, Bowen and Walters \cite{Bo,Ru} that there exists a constant $C(\varphi)$ (independent of $X$), as well as a unique ergodic measure $\mu_\varphi$ on $(X,\sigma)$, namely the equilibrium state or Gibbs measure of $\varphi$ restricted to $X$, such that for any $t\in X$, $n\ge 0$ and $t'\in [t|_n]$,
\begin{equation}\label{gibbs}
C(\varphi)^{-1}\le \frac{\mu_\varphi([t|_n])}{\exp(S_n\varphi(t')-nP_X(\varphi))} \le C(\varphi),
\end{equation}
and $\mu_\varphi$ possesses the  quasi-Bernoulli property,  
\begin{equation}\label{quasib}
C(\varphi)^{-1}\mu_\varphi([w])\mu_\varphi([u])\le \mu_\varphi([wu]) \le C(\varphi) \mu_\varphi([w])\mu_\varphi([u]), \quad \forall w,u\in\Sigma_*, \ [wu]\cap X\neq\emptyset.
\end{equation}

\subsection{Multifractal analysis of Gibbs measure}\label{magibbs}

Here we follow \cite{Rand89,BBP}. Consider a topologically transitive subshift $X$ of finite type and a H\"older potential $\varphi$ on $X$. Denote by $\mu_\varphi$ the equilibrium state on $(X,\sigma)$ with potential $\varphi$.

Define the R\'enyi entropy or $L^q$ spectrum of $\mu_\varphi$ as
\begin{equation}\label{tau1}
\tau_{\mu_\varphi}(q)=\liminf_{n\to\infty}-\frac{1}{n}\log_2 \sum_{w\in \Sigma_n, \mu_\varphi([w])\neq 0} \mu_\varphi([w])^q,\ q\in\mathbb{R}.
\end{equation}
It is easy to deduce from~\eqref{pressure} and~\eqref{gibbs} that the above limit inferior is in fact a limit, and it is equal to
\begin{equation}\label{tau2}
\tau_{\mu_\varphi}(q)=\frac{1}{\log 2} (qP_X(\varphi)-P_X(q\varphi)).
\end{equation}
Due to Corollary 5.27 in~\cite{Ru}, if $(X,\sigma)$ is topologically transitive, then $q\mapsto P_X(q\varphi)$ is a convex analytic function on $\mathbb{R}$, thus $\tau_{\mu_\varphi}$ is a concave analytic function on $\mathbb{R}$ and
\begin{equation}\label{tau3}
\tau_{\mu_\varphi}'(q)=\frac{1}{\log 2} (P_X(\varphi)-\frac{\mathrm{d}}{\mathrm{d} q}P_X(q\varphi)).
\end{equation}

Denote by $\tau_{\mu_\varphi}^*: \alpha\in\mathbb{R}\mapsto \inf_{q\in\mathbb{R}} q\alpha-\tau_{\mu_\varphi}(q)$ the Legendre transform of $\tau_{\mu_\varphi}$. Since $\tau_{\mu_\varphi}$ is concave and analytic over $\mathbb{R}$, we have for any $q\in \mathbb{R}$,
\begin{equation}\label{tau4}
\tau^*_{\mu_\varphi}(\tau_{\mu_\varphi}'(q))=q\tau_{\mu_\varphi}'(q)-\tau_{\mu_\varphi}(q)=-q\frac{\mathrm{d}}{\mathrm{d} q}P_X(q\varphi)+P_X(q\varphi).
\end{equation}

For $\alpha\in \{\tau_{\mu_\varphi}'(q): q\in\mathbb{R}\}$ define the set
\[
E_{\mu_\varphi}(\alpha)=\Big \{t\in X: \lim_{n\to\infty} \frac{\log \mu_\varphi([t|_n])}{\log 2^{-n}}=\alpha\Big \}
\]
and
\[
\widetilde{E}_{\mu_\varphi}(\alpha)=\Big \{t\in X: \lim_{n\to\infty} \frac{\log \mu_\varphi([t|_n])}{\log 2^{-n}}=\lim_{n\to\infty} \frac{\log \max_{w\in\mathcal{N}(t|_n)}\mu_\varphi([w])}{\log 2^{-n}}=\alpha\Big \},
\]
where for any $w\in\Sigma_*$ define the set of neighbor words of $w$ by
\begin{equation}\label{neighbor}
\mathcal{N}(w)=\Big\{u\in\Sigma_{|w|}: \big|\sum_{i=1}^{|w|} (u_i-w_i) \cdot 2^{-i}\big| \le 2^{-|w|}\Big\}.
\end{equation}

By using~\eqref{tau1}, it is standard to check that
\begin{equation*}
\dim_H \widetilde{E}_{\mu_\varphi}(\alpha) \le \dim_H E_{\mu_\varphi}(\alpha) \le \tau^*_{\mu_\varphi}(\alpha).
\end{equation*}
Moreover, one can prove that this is actually an equality: For $q\in\mathbb{R}$ denote by $\mu_{q\varphi}$ the equilibrium state of $q\varphi$ restricted to $X$. Then applying~\eqref{gibbs} to $\mu_\varphi$ and $\mu_{q\varphi}$, together with~\eqref{tau2} we can easily get for any $t\in X$ and $n\ge 1$,
\[
(C(\varphi)^{|q|}C(q\varphi))^{-1}\mu_\varphi([t|_n])^q e^{-n\tau_{\varphi}(q)} \le \mu_{q\varphi}([t|_n]) \le C(\varphi)^{|q|}C(q\varphi)\mu_\varphi([t|_n])^q e^{-n\tau_{\varphi}(q)}.
\]
Then due to~\cite{BBP} and the fact that $\mu_\varphi$ is quasi-Bernoulli, we get for $\mu_{q\varphi}$-almost every $t\in X$,
\[
\lim_{n\to\infty} \frac{\log \mu_{\varphi}([t|_n])}{\log 2^{-n}}=\lim_{n\to\infty} \frac{\log \max_{w\in\mathcal{N}(t|_n)}\mu_\varphi([w])}{\log 2^{-n}}=\tau_{\mu_\varphi}'(q)\]
and
\[\lim_{n\to\infty} \frac{\log \mu_{q\varphi}([t|_n])}{\log 2^{-n}}=\tau^*_{\mu_\varphi}(\tau_{\mu_\varphi}'(q)).
\]
Due to the mass distribution principle, this implies that for any $q\in \mathbb{R}$,
\begin{equation}\label{mulfor}
\dim_H \widetilde{E}_{\mu_\varphi}(\tau_{\mu_\varphi}'(q))\ge \dim_H (\mu_{q\varphi}) \ge \tau^*_{\mu_\varphi}(\tau_{\mu_\varphi}'(q)),
\end{equation}
where for any positive Borel measure $\mu$ defined on a compact metric space, the lower Hausdorff dimension of $\mu$ is given by $\dim_H (\mu) =\inf \{ \dim_H E: \mu(E)>0\}$.

\section{Proof of Theorem~\ref{mainthm}}\label{prfmainthm}

From now on we fix a H\"older potential $\varphi$ on the $\Sigma$ and denote by $\mu$ the Gibbs measure on $(\Sigma,\sigma)$ with potential $\varphi$. We avoid the trivial case that $\varphi$ is a constant function.

\subsection{The multifractal nature of $F_\mu$ and $F_\mu^{\text{pert}}$. }\label{randomspec}

Denote the canonical mapping 
\[
\lambda: w\in\Sigma_*\bigcup\Sigma\mapsto \lambda(w)=\sum_{i=1}^{|w|} w_i \cdot 2^{-i}\in [0,1].
\]

For $w\in\Sigma_*$ let
\[
T_w:x\in\mathbb{R}\mapsto 2^{-|w|}\cdot x+\lambda(w) \text{ and }\psi_w=\psi\circ T_w^{-1}.
\]

Consider the wavelet series
\begin{equation}\label{Fpert}
F_\mu(x)=\sum_{w\in\Sigma_*} d_w\cdot \psi_w(x), \text{ with }|d_w|=2^{-|w|(s_0-1/p_0)}\mu([w])^{1/p_0}.
\end{equation}
Up to the formal replacement of dyadic intervals by the cylinders of $\Sigma$, this is the wavelet series built from the image of $\mu$ by $\lambda$ in Section~\ref{1.3}.  
Recall (see \eqref{xisp}) that
\begin{equation*}
\xi_{F_\mu}(q)=q(s_0-1/p_0)+\tau_\mu(q/p_0), \ q\in\mathbb{R}.
\end{equation*}
So for $\alpha\ge 0$ and $h=s_0-1/p_0+\alpha/p_0$ such that $h\le r_0$, we have
\begin{equation}\label{xiup2}
\dim_H \{x\in (0,1): \bar{h}_{F_\mu}(x)=h\} \le \xi_{F_\mu}^*(h) = \tau_\mu^*(\alpha).
\end{equation}

For $q\in \mathbb{R}$ denote by $\mu_q$ the equilibrium state of the potential $q\varphi$ on $(\Sigma,\sigma)$.. Applying the results in Section~\ref{magibbs} we have for any $q\in\mathbb{R}$, for $\mu_q$-almost every $t\in\Sigma$,
\[
\lim_{n\to\infty}\frac{\log \mu([t|_n])}{\log 2^{-n}}=\lim_{n\to\infty} \frac{\log \max_{w\in\mathcal{N}(t|_n)}\mu([w])}{\log 2^{-n}}=\tau'_\mu(q).
\]
Together with Theorem 1 of~\cite{BSw}, this implies that $\mu_q$ is carried by $\{x\in (0,1): \bar{h}_{F_\mu}(x)=s_0-1/p_0+\tau_\mu'(q)/p_0\}$. Then due to~\eqref{xiup2} and~\eqref{mulfor}, we get that $F_\mu$ obeys the multifractal formalism at each $h=s_0-1/p_0+\alpha/p_0$ such that $h\le r_0$ and $\alpha\in\{\tau_\mu'(q):q\in\mathbb{R}\}$:
\[
\dim_H \{x\in [0,1]: \bar{h}_{F_\mu}(x)=h\}=\xi_{F_\mu}^*(h)=\tau^*_\mu(\alpha).
\]

\medskip

The random perturbation $F_\mu^{\text{pert}}$ is obtained from $F_\mu$ and a sequence of independent random variables $\{\pi_w\}_{w\in \Sigma_*}$ as 
\[
F_\mu^{\text{pert}}(x)=\sum_{w\in\Sigma_*} \pi_w\cdot d_w\cdot \psi_w(x),
\]
and our assumption (A1) is: For any $q\in\mathbb{R}$ we have $\sup_{w\in\Sigma_*} \mathbb{E}(|\pi_w|^q)<\infty$. We have seen in Section~\ref{1.3} that this implies that
\begin{equation}\label{perturb}
\xi_{F_\mu^{\text{pert}}}=\xi_{F_\mu} \text{ and } \bar{h}_{F_\mu^{\text{pert}}}=\bar{h}_{F_\mu} \text{ over $(0,1)$ almost surely}.
\end{equation}
Thus, $F_\mu^{\text{pert}}$ fulfills the multifractal formalism at $h$, whenever $F_\mu$ does.

\subsection{Topologically transitive subshifts of finite type avoiding the set of zeros of $\psi$}

For $k\ge 0$ and $x\in[0,1)$, let $x|_k$ be the unique word $w\in\Sigma_k$ such that
\[
\lambda(w)\le x< \lambda(w)+2^{-k},
\]
as well as $1|_k=1\cdots 1$ for $k\ge 1$.

\medskip

Let $\mathcal{Z}=\psi^{-1}(\{0\})\cap[0,1]$. We have assumed that $\mathcal{Z}$ is finite ((A2)).

\medskip

For $k\ge 2$ define the set of forbidden words by
\[
\mathcal{F}_k=\bigcup_{x\in \mathcal{Z}} \mathcal{F}_k(x),
\]
where
\[
\mathcal{F}_k(x)=\left\{\begin{array}{ll}
\{x|_k\}, & \text{if } x\not\in \lambda(\Sigma_*), \\
\{w\in\Sigma_k: 0\le \lambda(x|_k)-\lambda(w)\le 2^{-k}\}, & \text{otherwise}.
\end{array}
\right.
\]

Define the subshift of finite type with respect to $\mathcal{F}_k$ by
\[
X_k=\{t\in \Sigma: \sigma^m(t)|_k\not\in \mathcal{F}_k,\ \forall \ m\ge 0\}.
\]
Clearly for small $k$, the subshift $X_k$ might be a empty set. But, since $\mathcal{Z}$ is a finite set, it is easy to see that $X_k$ is not empty for all $k$ large enough. In fact, denote by $\delta=\min\{|x-y|: x,y\in\mathcal{Z},\ x\neq y\}>0$ and $k_0=[-\log_2\delta]+3$. Then for any $x,y\in\mathcal{Z}$ with $x<y$, there exists at least one word $w\in\Sigma_{k_0-1}$ such that $x<\lambda(w)<y$ thus $\lambda(x|_{k_0-1})<\lambda(w)<\lambda(y|_{k_0-1})$, since $y-x\ge \delta\ge 2^{-(k_0-2)}$. This ensures that for $k\ge k_0$, for all $w\in \mathcal{F}_k$,
%
his brother $w'$ (the unique $w'\in \Sigma_k$ such that $w'|_{k-1}=w|_{k-1}$, $w'\neq w$) is an admissible word. Thus for any $u\in\Sigma_{k-1}$, at least one of $u0, u1$ is allowed in $X_k$, which also implies that for each $u\in\Sigma_{k-1}$, there exists an infinite word $t\in X_k$ such that $t|_k=u0$ or $t|_k=u1$. So for $k\ge k_0$, the Hausdorff distance between $X_k$ and $\Sigma$ is not greater than $2^{-k}$, that is
\begin{equation}\label{Hdist}
\mathrm{dist}_H(X_k,\Sigma):=\max\{ \sup_{s\in X_k} \inf_{t\in\Sigma}\rho(s,t), \sup_{s\in\Sigma}\inf_{t\in X_k} \rho(s,t)\} \le 2^{-k},
\end{equation}
thus it converges to $0$ when $k\to\infty$.

Since $X_k$ is a increasing sequence (it is easy to see that $\Sigma\setminus X_k \supset \Sigma\setminus X_{k+1}$),  $\overline{\dim}_B X_k$ increases and converges to $1$ as $k\to\infty$. Otherwise  $\overline{\dim}_B \bigcup_{k} X_k <1$, thus $\bigcup_{k} X_k$ is not dense in $\Sigma$, which is in contradiction  with~\eqref{Hdist}. Here $\overline{\dim}_B$ is the upper box-counting dimension (see~\cite{Falc} for the definition and properties).

It is known that any subshift of finite type can be decomposed into several disjoint closed sets $X_{k,1},\cdots,X_{k,m}$, $m\ge 1$, and each of them is a topologically transitive subshift of finite type. This can be deduced from the non-negative matrix analysis that one can always decomposes reducible matrix into several irreducible pieces.

The finite stability of $\overline{\dim}_B$ (see Section 3.2 in~\cite{Falc}) implies
\[
\overline{\dim}_B X_k=\max_{i=1,\cdots,m} \overline{\dim}_B X_{k,i},
\]
so we can choose one of the $X_{k,i}$ such that $\overline{\dim}_B X_{k,i}=\overline{\dim}_BX_{k}$ and also denote it as $X_k$. Then we obtain a sequence of topologically transitive subshift of finite type $(X_k)_{k\ge 1}$ such that the upper box-counting dimension $\overline{\dim}_B X_k$ converges to $1$. We prove that this sequence converge to $\Sigma$ in the Hausdorff distance: 

\smallskip

Suppose that  it is not the case, then there exist an $\epsilon>0$ and a subsequence $(X_{k_j})_{j\ge 1}$ such that $\mathrm{dist}_H(X_{k_j},\Sigma)\ge \epsilon$ for $j\ge 1$. Fix an integer $N> -\log_2\epsilon+1$. Then $\mathrm{dist}_H(X_{k_j},\Sigma)\ge \epsilon$ implies that there exist a $w_j\in\Sigma_N$ such that $X_{k_j} \cap [w_j]=\emptyset$. Since $\#\Sigma_N =2^N$ is finite, then there exist $w_*\in \Sigma_N$ and a subsequence $(X_{k_j'})_{j\ge 1}$ of $(X_{k_j})_{j\ge 1}$ such that $X_{k_j'}\cap [w_*]=\emptyset$ for $j\ge 1$.

Since $X_{k'_j}$ is a subshift of finite type,  $X_{k_j'}\cap [w_*]=\emptyset$ implies that
\[
X_{k_j'} \subset X_*:=\{t\in\Sigma: \sigma^m(t)|_N\neq w_*,\ \forall \ m\ge 0\}.
\]
Denote by $B_*$ the transition matrix of $X_*$ and $\lambda_*$ the maximal eigenvalue of $B_*$. Due to the standard Perron-Frobenius theory (\cite{Sen}, Thm 1.1),  $\lambda_*$ is strictly less than the maximal eigenvalue of the transition matrix of the full shift, which is equal to $2$. This yields that $\overline{\dim}_B X_* =\log \lambda_*/\log 2 <1$, which is in contradiction with the fact that $\overline{\dim}_B X_*\ge \lim_{j\to\infty} \overline{\dim}_B X_{k_j'}=1$.

\smallskip

To end this section, since $\psi$ is $r_0$-smooth, for each $k\ge k_0$ we can easily find a constant $c_{\psi,k}>0$ such that for each $t\in X_k$,
\begin{equation}\label{cpsi}
|\psi(\lambda(\sigma^m(t)|_n))|\ge c_{\psi,k},\ \forall\ m\ge 0,\ n\ge k.
\end{equation}
This is the main property required in our proof, which clearly would not hold if we considered any $t\in\Sigma$.

\subsection{Lower bound estimation}\label{Lbe}

For $k\ge k_0$ and $q\in\mathbb{R}$, denote by $\mu_q^{(k)}$ the Gibbs measure on $(X_k,\sigma)$ with potential $q\varphi$.

Apply~\eqref{gibbs} both to $\mu^{(k)}_q$ and $\mu$, together with~\eqref{tau2} we have for any $t\in X_k$ and $n\ge 1$,
\begin{multline*}
(C(\varphi)^{|q|}C(q\varphi))^{-1} \mu([t|_n])^q e^{-n\tau_{\varphi}(q)} e^{-n\frac{P_\Sigma(q\varphi)-P_{X_k}(q\varphi)}{\log2}} \le \mu_{q}^{(k)}([t|_n]) \\\le C(\varphi)^{|q|}C(q\varphi) \mu([t|_n])^q e^{-n\tau_{\varphi}(q)}e^{-n\frac{P_\Sigma(q\varphi)-P_{X_k}(q\varphi)}{\log2}}.
\end{multline*}
By using large deviation method as in~\cite{BBP}, it is standard to prove that for $\mu_q^{(k)}$-almost every $t\in X_k$,
\begin{eqnarray}
\lim_{n\to\infty} \frac{\log \mu([t|_n])}{\log 2^{-n}}
&=& \lim_{n\to\infty} \frac{\log \max_{w\in\mathcal{N}(t|_n)}\mu([w])}{\log 2^{-n}} \nonumber \\
&=& \tau'_\mu(q)+\frac{\mathrm{d}}{\mathrm{d} q} \frac{P_\Sigma(q\varphi)-P_{X_k}(q\varphi)}{\log2}:= \alpha^{(k)}_q, \label{alphakq}
\end{eqnarray}
and
\begin{eqnarray}
&&\lim_{n\to\infty} \frac{\log \mu^{(k)}_q([t|_n])}{\log 2^{-n}} \nonumber \\
&=& \tau^*_\mu(\tau'_\mu(q))-\frac{-q\frac{\mathrm{d}}{\mathrm{d} q}(P_{\Sigma}(q\varphi)-P_{X_k}(q\varphi))+(P_\Sigma(q\varphi)-P_{X_k}(q\varphi))}{\log 2}:=D^{(k)}_q. \label{dkq}
\end{eqnarray}

Let $h^{(k)}_q=s_0-1/p_0+\alpha^{(k)}_q/p_0$. Then the above two equations together with Theorem 1 of~\cite{BSw} and~\eqref{perturb} imply that
\begin{equation}\label{mukq}
\mu_q^{(k)} \text{ is carried by } \{x\in [0,1]: \bar{h}_{F^{\text{pert}}_\mu}(x)=h^{(k)}_q\} \text{ and } \dim_H (\mu_q^{(k)}) \ge D^{(k)}_q.
\end{equation}

We deduce from $\mu_q^{(k)}$ two Borel measures $\mu_{q,G}^{(k)}$, $\mu_{q,R}^{(k)}$ carried by the graph and range of $F^{\text{pert}}_\mu$ respectively in the following way:
\begin{itemize}
\item For any Borel set $A\subset G_{F^{\text{pert}}_\mu}([0,1])$, let
\[
\mu_{q,G}^{(k)}(A)=\mu_q^{(k)}\left(t\in X_k : (\lambda(t),F^{\text{pert}}_\mu(\lambda(t)))\in A\right);
\]
\item For any Borel set $B\subset R_{F^{\text{pert}}_\mu}([0,1])$, we have
\[
\mu_{q,R}^{(k)}(B)=\mu_q^{(k)}\left(t\in X_k: F^{\text{pert}}_\mu(\lambda(t))\in B\right).
\]
\end{itemize}

As the essential intermediate result of this paper, we have the following theorem.
\begin{theorem}\label{keythm}
With probability $1$ for all $q\in\mathbb{R}$ with $0<h^{(k)}_q <1$, we have
\begin{eqnarray*}
\dim_H (\mu_{q,G}^{(k)}) &\ge & \gamma^{(k)}_{q,G}:=\frac{D^{(k)}_q}{h^{(k)}_q} \wedge \Big(1-h^{(k)}_q+D^{(k)}_q\Big), \\
\dim_H (\mu_{q,R}^{(k)}) &\ge &  \gamma^{(k)}_{q,R}:=\frac{D^{(k)}_q}{h^{(k)}_q} \wedge 1.
\end{eqnarray*}
\end{theorem}
Let us show how it makes it possible to conclude.

For $k\ge k_0$ let $I^{(k)}=\{h_q^{(k)}:q\in\mathbb{R}\}\cap (0,1)$ and $J^{(k)}=\bigcap_{p\ge k} I^{(p)}$. Also for $S\in\{G,R\}$ define the function
\[
f^{(k)}_S: h_q^{(k)}\in I^{(k)}\mapsto \gamma^{(k)}_{q,S}.
\]
Due to~\eqref{mukq}, $\mu_{q,S}^{(k)}$ is carried by $S_{F^{\text{pert}}_\mu}(h^{(k)}_q)$. Thus Theorem~\ref{keythm} implies that, for each $k\ge k_0$, with probability 1 for all $h\in J^{(k)}$,
\[
d^S_{F^{\text{pert}}_\mu}(h)\ge f^{(p)}_S(h), \ \forall\ p\ge k.
\]
Then to end the proof of Theorem~\ref{mainthm}, it only remains to show that for each $q\in\mathbb{R}$,
\begin{equation}\label{hdconv}
h^{(k)}_q \to h_q=s_0-1/p_0+\tau'_\mu(q)/p_0 \text{ and } D^{(k)}_q \to D_q= \tau^*_\mu(\tau_\mu'(q)) \text{ as } k\to\infty,
\end{equation}
which implies that $\bigcup_{k\ge k_0}J^{(k)} =\{s_0-1/p_0+\tau'_\mu(q)/p_0:q\in\mathbb{R}\}\cap (0,1)$ and for any compact subset $I\subset \bigcup_{k\ge k_0}J^{(k)}$, the functions $f^{(k)}_S$, $S\in\{G,R\}$ restricted to $I$ converge uniformly to
\[
f_S: h_q\in I \mapsto \gamma_{q,S},
\]
where
\[
\gamma_{q,G}:=\frac{D_q}{h_q} \wedge \Big(1-h_q+D_q\Big) \ \text{ and } \ \gamma_{q,R}:=\frac{D_q}{h_q} \wedge 1.
\]
This implies that with probability 1 for all $h\in I$ and $\alpha=hp_0+1-s_0p_0$,
\[
d^G_{F^{\text{pert}}_\mu}(h) \ge \frac{\tau^*_\mu(\alpha)}{h} \wedge \Big(1-h+\tau^*_\mu(\alpha)\Big)  \ \text{ and }\ d^R_{F^{\text{pert}}_\mu}(h) \ge \frac{\tau^*_\mu(\alpha)}{h} \wedge 1.
\]
Together with Theorem A and Section~\ref{randomspec}, we get the conclusion by taking a sequence of $I$ converging to $\{s_0-1/p_0+\tau'_\mu(q)/p_0:q\in\mathbb{R}\}\cap (0,1)$.

\smallskip

Now we prove~\eqref{hdconv}. This can be done due to~\eqref{alphakq},~\eqref{dkq} and the following lemma:

\begin{lemma}
Given $q\in\mathbb{R}$, we have $\lim_{k\to\infty} P_{X_k}(q\varphi)=P_\Sigma(q\varphi)$. Consequently, since these functions are convex and analytic, $P_{X_k}(q\phi)$ and $\frac{\mathrm{d}}{\mathrm{d}q}P_{X_k}(q\phi)$ converge uniformly on compact intervals to $P_{\Sigma}(q\phi)$ and $\frac{\mathrm{d}}{\mathrm{d}q}P_{\Sigma}(q\phi)$ respectively.
\end{lemma}

\begin{proof}

The idea is borrowed from the proof of Proposition 2 in~\cite{GeMi}.

Assume that this is not the case for some $q\in\mathbb{R}$. Since $P_{X_k}(q\varphi) \le P_\Sigma(q\varphi)$, let $P_*(q\varphi)=\liminf_{k\to\infty} P_{X_k}(q\varphi)$ and let $\delta=P_\Sigma(q\varphi)-P_*(q\varphi)>0$.

Take a subsequence $(\mu^{(k_j)}_q)_{j\ge1}$ converging to some probability measure $\mu^*_q$ in the $\text{weak}^*$ topology. Due to Theorem B, for any $t\in\Sigma$ and $n\ge 1$,
\[
(C(\varphi)^{|q|}C(q\varphi))^{-1} \le \frac{\mu_q([t|_n])}{\exp(S_nq\varphi(t)-nP_\Sigma(q\varphi))} \le (C(\varphi)^{|q|}C(q\varphi)).
\]
Since $X_k$ converges to $\Sigma$ in sense of Hausdorff distance, then for all $k$ large enough, we have $X_k\cap [t|_n]\neq \emptyset$, thus
\[
(C(\varphi)^{|q|}C(q\varphi))^{-1} \le \frac{\mu_q^{(k)}([t|_n])}{\exp(S_nq\varphi(t)-nP_{X_k}(q\varphi))} \le (C(\varphi)^{|q|}C(q\varphi)).
\]
This implies
\[
\mu_q([t|_n])\le \mu^*_q([t|_n]) \cdot (C(\varphi)^{|q|}C(q\varphi))^{2}\cdot \exp(-n\delta/2).
\]
Taking $n$ large enough sothat $(C(\varphi)^{|q|}C(q\varphi))^{2}\cdot \exp(-n\delta/2)<1$, this is in contradiction with the fact that both $\mu_q$ and $\mu_q^*$ are probability measures.

\end{proof}

\section{Proof of Theorem~\ref{keythm}}\label{prfkeythm}

From now on we fix a $k\ge k_0$ such that $X_k\neq \emptyset$. 

\subsection{Main proof}

\begin{proof}

Recall~\eqref{neighbor} that the set of neighbor words of $w\in\Sigma_*$ is
\begin{equation*}
\mathcal{N}(w)=\Big\{u\in\Sigma_{|w|}: \big|\sum_{i=1}^{|w|} (u_i-w_i) \cdot 2^{-i}\big| \le 2^{-|w|}\Big\}.
\end{equation*}

For $p\ge 1$ let $\mathcal{P}_p$ be the subset of pairs of elements of $\Sigma_{p+1}$ defined as
\begin{equation}
\mathcal{P}_p=\{(u,v)\in\Sigma_{p+1}\times \Sigma_{p+1}: v|_p\in\mathcal{N}(u|_p),\ v\not\in \mathcal{N}(u)\}.
\end{equation}
Then for any $s,t\in\Sigma$ with $|s-t|>0$, there exists a unique $p\ge 1$ such that $\mathbf{1}_p(s,t)=1$, where the indicator function is defined by
\[
\mathbf{1}_p(s,t)=\left\{
\begin{array}{ll}
1, & \text{if } (s|_{p+1},t|_{p+1})\in \mathcal{P}_p;\\
0, & \text{otherwise}.
\end{array}
\right.
\]

By construction we know that if $\mathbf{1}_p(s,t)=1$, then
\[
\inf_{s'\in [s|_{p+1}],t'\in [t|_{p+1}]} |s'-t'|\ge 2^{-p-1} \text{ and } \sup_{s'\in [s|_{p+1}],t'\in [t|_{p+1}]} |s'-t'| \le 2^{-p+1}.
\]

Recall that for $w\in\Sigma_*$,
\[
T_w:x\in\mathbb{R}\mapsto 2^{-|w|}\cdot x+\lambda(w).
\]
Then for any $s,t\in\Sigma$ with $\mathbf{1}_p(s,t)=1$, for any $m\ge 1$ we have
\[
T_{s|_{p+1+m}}^{-1}(\lambda(s))\in [0,1]\ \text{ and }\ |T_{s|_{p+1+m}}^{-1}(\lambda(s))-T_{s|_{p+1+m}}^{-1}(\lambda(t))|\ge 2^m.
\]
Since $\psi$ decays at infinity, due to~\eqref{cpsi}, there exists a large enough $N_{\psi,k}\ge 1$ such that for any $s,t\in X_k$ with $\mathbf{1}_p(s,t)=1$, for any $m\ge 1$ and $n\ge k$
\begin{equation}\label{gec}
|\psi_{s|_{p+1+N_{\psi,k}+m}}(\lambda(s)|_{p+1+N_{\psi,k}+m+n})-\psi_{s|_{p+1+N_{\psi,k}+m}}(\lambda(t)|_{p+1+N_{\psi,k}+m+n})| \ge \frac{c_{\psi,k}}{2}.
\end{equation}

Let $J_k=\{q\in\mathbb{R}: 0<h^{(k)}_q <1\}$.

For any $q\in J_k$, $\epsilon>0$ and $w\in\Sigma_*$, define
\begin{eqnarray}
\mathbf{1}_{w}^{(a)}(q,\epsilon)&=& \mathbf{1}_{\left\{ |d_w|\in \big[2^{-|w|(h^{(k)}_q+\epsilon)},2^{-|w|(h^{(k)}_q-\epsilon)}\big] \right\}} \label{indicatora} \\
\mathbf{1}_{w}^{(b)}(q,\epsilon)&=& \mathbf{1}_{\left\{ \mu^{(k)}_q([w])\in \big[2^{-|w|(D^{(k)}_q+\epsilon)},2^{-|w|(D^{(k)}_q-\epsilon)}\big] \right\}} \label{indicatorb} \\
\mathbf{1}_{w}^{(c)}(q,\epsilon)&=& \mathbf{1}_{\left\{\sup\limits_{s,t\in\bigcup_{u\in\mathcal{N}(w)} [u]} \big|F^{\text{pert}}_\mu(s)-F^{\text{pert}}_\mu(t)\big|\le 2^{-|w|(h^{(k)}_q-\epsilon)}\right\}}.\label{indicatorc}
\end{eqnarray}

For $n\ge 1$ define
\[
\Sigma^{(k)}_n(q,\epsilon)=\left\{w\in\Sigma_n: [w]\cap X_k\neq \emptyset \text{ and } \mathbf{1}_{w}^{(a)}(q,\epsilon)\cdot\mathbf{1}_{w}^{(b)}(q,\epsilon)\cdot\mathbf{1}_{w}^{(c)}(q,\epsilon)=1 \right\}.
\]
(In fact, $\mathbf{1}_{w}^{(b)}(q,\epsilon)=1$ implies $[w]\cap X_k\neq \emptyset$). Then let
\[
E_n^{(k)}(q,\epsilon)=\bigcap_{p\ge n} \bigcup_{w\in\Sigma_n^{(k)}(q,\epsilon)} [w] \ \text{ and }\ 
E^{(k)}(q)=\lim_{\epsilon\to 0^+}\lim_{n\to\infty} E_n^{(k)}(q,\epsilon).
\]
Due to~\eqref{perturb},~\eqref{alphakq} and~\eqref{dkq}, with probability 1, for all $q\in J_k$,
\begin{equation}\label{CEfh1}
E^{(k)}(q)\subset X_k\cap E_{F^{\text{pert}}_\mu}(h^{(k)}_q) \ \text{ and }\ \mu_q^{(k)}(E^{(k)}(q)) =1.
\end{equation}

For $\gamma>0$ define the Riesz-like kernel: for $s,t\in\Sigma_*\bigcup\Sigma$,
\begin{equation}\label{kernel}
\mathcal{K}_{\gamma}(s,t)=\left\{
\begin{array}{ll}
(|F^{\text{pert}}_\mu(\lambda(s))-F^{\text{pert}}_\mu(\lambda(t))|^2+|\lambda(s)-\lambda(t)|^2)^{-\frac{\gamma}{2}} \vee 1, & \text{if } \gamma\ge 1,\\
\ \\
|F^{\text{pert}}_\mu(\lambda(s))-F^{\text{pert}}_\mu(\lambda(t))|^{-\gamma} \vee 1, & \text{if } \gamma< 1.
\end{array}
\right.
\end{equation}

For $q\in J_k$ recall that
\begin{equation}\label{gammaGR}
\gamma_{q,G}^{(k)}= \frac{D^{(k)}_q}{h^{(k)}_q} \wedge \big(D^{(k)}_q+1-h^{(k)}_q\big) \ \text{ and }\ \gamma_{q,R}^{(k)}=\frac{D^{(k)}_q}{h^{(k)}_q} \wedge 1.
\end{equation}

For $q\in J_k$, $\delta>0$ and $\epsilon>0$ we define the $n^{\text{th}}$-energy for $n\ge 1$ and $S\in\{G,R\}$:
\[
\mathcal{I}_{n,\delta}^{S}(q,\epsilon)=\iint_{s,t\in E_n^{(k)}(q,\epsilon),s\neq t} \mathcal{K}_{\gamma_{q,S}^{(k)}-\delta}(s,t)\ \mathrm{d}\mu_q^{(k)}(s) \mathrm{d}\mu_q^{(k)}(t).
\]

Let $K$ be any compact subinterval of $J_k$. We assume for a while that we have proved that for any $\delta$ small enough, there exists $\epsilon_\delta>0$ such that for any $n\ge 1$, $\epsilon\in (0,\epsilon_\delta)$ and $S\in \{G, R\}$,
\begin{equation}\label{IGR}
\mathbb{E}\left(\sup_{q\in K} \mathcal{I}^{S}_{n,\delta}(q,\epsilon) \right)<\infty.
\end{equation}
The following lemma is a slight modification of Theorem 4.13 in~\cite{Falc} regarding the Hausdorff dimension estimate through the potential theoretic method.

\begin{lemma}\label{ptm}
Let $\nu$ be a Borel measure on $\mathbb{R}^{d}$ and let $E\subset\mathbb{R}^{d}$ be a Borel set such that $\nu(E)>0$. For any $\gamma>0$, if 
$$
\iint_{x,y\in E, x\neq y} |x-y|^{-\gamma}\vee 1 \ \mathrm{d}\nu(x)\mathrm{d}\nu(y)<\infty,
$$
then
$$\nu\left(\left\{ x\in E : \underline{\dim}_{\mathrm{loc}}\nu(x)= \liminf_{r\to 0^+} \frac{\log \nu(B(x,r))}{\log r}< \gamma\right\}\right)=0.$$
\end{lemma}

Then, it easily follows from~\eqref{IGR} and Lemma~\ref{ptm} that, with probability 1, for all $q\in K$:
\begin{itemize}
\item For $\mu_{q,G}^{(k)}$-almost every $x\in\{(\lambda(t),F^{\text{pert}}_\mu(\lambda(t))): t\in E^{(k)}(q) \}$,
$$
\underline{\dim}_{\mathrm{loc}}\mu_{q,G}^{(k)}(x) \ge \gamma_{q,G}^{(k)}-\delta;
$$
\item For $\mu_{q,R}^{(k)}$-almost every $y\in\{F^{\text{pert}}_\mu(\lambda(t)): t \in E^{(k)}(q) \}$,
$$
\underline{\dim}_{\mathrm{loc}}\mu_{q,R}^{(k)}(y) \ge \gamma_{q,R}^{(k)}-\delta.
$$
\end{itemize}

Since $\delta$ can be taken arbitrarily small,  we get the conclusion by taking a countable sequence of compact subintervals $K_{j}\subset J_k$ such that $\bigcup K_{j}=J_k$.

\medskip

Now we prove \eqref{IGR}. 

\medskip

For any $\bar{q}\in K$ and $\epsilon>0$ we define the neighborhood of $\bar{q}$ in $K$:
\begin{equation}\label{Ulambda}
U_\epsilon(\bar{q})=\left\{q\in K: \max \left\{ \begin{array}{l}  |q-\bar{q}|, |\alpha^{(k)}_q-\alpha^{(k)}_{\bar{q}}|, |h^{(k)}_q-h^{(k)}_{\bar{q}}|, \\
|D^{(k)}_q-D^{(k)}_{\bar{q}}|,|\gamma_{q,G}^{(k)}-\gamma_{\bar{q},G}^{(k)}|, |\gamma_{q,R}^{(k)}-\gamma_{\bar{q},R}^{(k)}| \end{array} \right\} < \epsilon \right\}.
\end{equation}
By continuity of these functions, the set $U_\epsilon(\bar{q})$ is open in $K$.
 
Notice that for $q\in K$, $\delta>0$ and $S\in\{G,R\}$ the Riesz-like kernels $\mathcal{K}_{\gamma_{q,S}^{(k)}-\delta}$ are positive functions and, moreover, by the continuity of $F^{\text{pert}}_{\mu}$ we have for any $s,t\in \Sigma$,
\[
\lim_{m\to \infty} \mathcal{K}_{\gamma_{q,S}^{(k)}-\delta} (s|_m,t|_m)=\mathcal{K}_{\gamma_{q,S}^{(k)}-\delta}(s,t).
\]
Then by applying Fatou's lemma we get for any $q\in U_\epsilon(\bar{q})$,
\begin{eqnarray*}
&&\mathcal{I}^{S}_{n,\delta}(q,\epsilon) \\
&=& \iint_{s,t\in E_n^{(k)}(q,\epsilon),s\neq t} \lim_{m\to \infty} \mathcal{K}_{\gamma_{q,S}^{(k)}-\delta} (s|_m,t|_m) \ \mathrm{d} \mu_q^{(k)}(s) \mathrm{d} \mu_q^{(k)}(t) \\
&=&\sum_{p\ge 1} \iint_{s,t\in E_n^{(k)}(q,\epsilon); \atop \mathbf{1}_p(s,t)=1} \lim_{m\to \infty} \mathcal{K}_{\gamma_{q,S}^{(k)}-\delta} (s|_m,t|_m) \  \mathrm{d} \mu_q^{(k)}(s) \mathrm{d} \mu_q^{(k)}(t) \\
&\le& \sum_{p\ge 1} \liminf_{m\to \infty}  \iint_{s,t\in E_n^{(k)}(q,\epsilon); \atop \mathbf{1}_p(s,t)=1}  \mathcal{K}_{\gamma_{q,S}^{(k)}-\delta} (s|_m,t|_m)\ \mathrm{d} \mu_q^{(k)}(s) \mathrm{d} \mu_q^{(k)}(t) \\
&= &\sum_{p\ge 1} \liminf_{m\to \infty} \sum_{u,v\in \Sigma_m; \atop \mathbf{1}_p(u,v)=1} \mathcal{K}_{\gamma_{q,S}^{(k)}-\delta} (u,v) \cdot \mu_q^{(k)}([u]\cap E_n^{(k)}(q,\epsilon)) \mu_q^{(k)}([v]\cap E_n^{(k)}(q,\epsilon))\\
&\le& \sum_{p\ge 1} \liminf_{m\to \infty} \sum_{u,v\in \Sigma_m; \atop \mathbf{1}_p(u,v)=1} \mathcal{K}_{\gamma_{\bar{q},S}^{(k)}-\delta-\epsilon} (u,v) \cdot \mu_q^{(k)}([u]\cap E_n^{(k)}(q,\epsilon)) \mu_q^{(k)}([v]\cap E_n^{(k)}(q,\epsilon)),
\end{eqnarray*}
where the last inequality comes from the fact that due to~\eqref{kernel} and~\eqref{Ulambda}, for any $\bar{q}\in K$, $\epsilon>0$ and $u,v\in \Sigma_*$, we have $\sup_{q\in U_\epsilon(\bar{q}) }  \mathcal{K}_{\gamma_{q,S}^{(k)}-\delta} (u,v) \le \mathcal{K}_{\gamma_{\bar{q},S}^{(k)}-\delta-\epsilon} (u,v)$.

Let 
$$
A_{p,m}=\sum_{u,v\in \Sigma_m; \atop \mathbf{1}_p(u,v)=1} \mathcal{K}_{\gamma_{\bar{q},S}^{(k)}-\delta-\epsilon} (u,v)\cdot \mu_q^{(k)}([u]\cap E_n^{(k)}(q,\epsilon)) \mu_q^{(k)}([v]\cap E_n^{(k)}(q,\epsilon)).
$$
Then,
\begin{eqnarray}
\sup_{q\in U_\epsilon(\bar{q}) }  \mathcal{I}^{S}_{n,\delta}(q,\epsilon) &\le& \sup_{q\in U_\epsilon(\bar{q}) } \sum_{p\ge 1} \liminf_{m\to \infty}\ A_{p,m}  \nonumber \\
&\le&  \sup_{q\in U_\epsilon(\bar{q}) } \sum_{p\ge 1} \Big( A_{p,m_p}+\sum_{m\ge m_p} |A_{p,m+1}-A_{p,m}|\Big) \nonumber \\ 
&\le&  \sum_{p\ge 1} \Big(  \sup_{q\in U_\epsilon(\bar{q}) }A_{p,m_p}+\sum_{m\ge m_p}  \sup_{q\in U_\epsilon(\bar{q}) }|A_{p,m+1}-A_{p,m}|\Big), \label{uniupper}
\end{eqnarray}
where for $p\ge 1$, we can choose $m_p\ge 2$ to be any integer. We have
\begin{equation}\label{IdeltaI}
\sup_{q\in U_\epsilon(\bar{q}) } A_{p,m} \le B_{p,m} \quad\text{and}\quad\sup_{q\in U_\epsilon(\bar{q}) }|A_{p,m+1}-A_{p,m}|\le \Delta B_{p,m},
\end{equation}
where
\begin{eqnarray*}
B_{p,m}
&=& \sum_{u,v\in \Sigma_m;\atop \mathbf{1}_p(u,v)=1} \mathcal{K}_{\gamma_{\bar{q},S}^{(k)}-\delta-\epsilon} (u,v) \sup_{q\in U_\epsilon(\bar{q}) } \mu_q^{(k)}([u]\cap E_n^{(k)}(q,\epsilon)) \mu_q^{(k)}([v]\cap E_n^{(k)}(q,\epsilon)),\\
\Delta B_{p,m}&=&\sum_{u,v\in \Sigma_m,u',v'\in \{0,1\}; \atop \mathbf{1}_p(u,v)=1} \big|\mathcal{K}_{\gamma_{\bar{q},S}^{(k)}-\delta-\epsilon} (uu',vv')-\mathcal{K}_{\gamma_{\bar{q},S}^{(k)}-\delta-\epsilon} (u,v)\big|\cdot \\
&&\qquad\qquad\qquad\qquad \sup_{q\in U_\epsilon(\bar{q}) } \mu_q^{(k)}([uu']\cap E_n^{(k)}(q,\epsilon)) \mu_q^{(k)}([vv']\cap E_n^{(k)}(q,\epsilon)),
\end{eqnarray*}
and we have used the equality $\mu_q^{(k)}([u]\cap E_n^{(k)}(q,\epsilon))=\sum_{u'\in\{0,1\}} \mu_q^{(k)}([uu']\cap E_n^{(k)}(q,\epsilon))$ to get the second inequality.

\begin{remark}{\rm \label{J12} $\ $
For technical reasons, we need to divide $J_k$ into two parts, in which $K$ will be chosen: 
\[
J_k'=\{q\in J_k : \gamma_{q,G}^{(k)}>1\} \text{ and } J_k''=\{q\in J_k : \gamma_{q,G}^{(k)} \le 1\}.
\]
Then, due to \eqref{gammaGR}, we have 
\begin{equation*}
\gamma_{q,G}^{(k)}=\left\{
\begin{array}{ll}
D^{(k)}_q+1-h^{(k)}_q & \textrm{if } q\in J_k',\\
\displaystyle D^{(k)}_q/h^{(k)}_q & \textrm{if } q\in J_k''
\end{array}
\right. \text{ and } 
\gamma_{q,R}^{(k)}=\left\{
\begin{array}{ll}
1& \textrm{if } q\in J_k',\\
\displaystyle D^{(k)}_q/h^{(k)}_q & \textrm{if } q\in J_k''.
\end{array}
\right.
\end{equation*}
}\end{remark}

For any compact subinterval $K$ of $J_k$ there exists $c_K>0$ such that for any $\epsilon<c_K$ and $q\in K$, $\gamma_{q,G}^{(k)}-\epsilon>1$ if $K\subset J_k'$ and $\gamma_{q,R}^{(k)}-\epsilon>0$ if $K\subset J_k''$.

Let $\delta_K=\epsilon_K=c_K/2$. We have the following key proposition:

\begin{proposition}\label{mainprop}
Let $S\in\{G,R\}$. Suppose that  $K$ is a compact subinterval of $J_k'$ or $J_k''$. For any $0<\delta<\delta_K$ we can find constants $c_1,c_2>0$, $\kappa_1,\kappa_2, \eta_1,\eta_2>0$ and $\epsilon_*>0$ such that for any $\bar{q}\in K$, $0<\epsilon\le \epsilon_*$, $n\ge 1$, $p\ge 1$, and $m\ge (p\vee n)+1+N_{\psi,k}+k$,
\begin{eqnarray*}
\mathbb{E}\left(B_{p,m}\right) &\le& c_1 \cdot C_{p\vee n} \cdot 2^n\cdot 2^{c_2\cdot ((p\vee n) -p)} \cdot 2^{-\eta_1\delta \cdot p+\kappa_1\epsilon\cdot m}; \\
\mathbb{E}\left(\Delta B_{p,m}\right) &\le& c_1 \cdot C_{p\vee n} \cdot 2^{c_2\cdot ((p\vee n) -p)} \cdot 2^{\kappa_2\cdot p-\eta_2\cdot m},
\end{eqnarray*}
where $C_{p\vee n}=\sup_{w\in\Sigma_{p\vee n+4+N_{\psi,k}}} \|f_w\|_\infty$, here $f_w$ is just the formal replacement of the bounded density function $f_{j,k}$ given in (A3).
\end{proposition}

Now fix any $n\ge 1$, and choose $m_p=\frac{\kappa_2+\frac{1}{2}\delta\eta_1}{\eta_2}\cdot (p\vee n)$ (by modifying a little $\eta_2$ we can always assume that $\frac{\kappa_2+\frac{1}{2}\delta\eta_1}{\eta_2}  (p\vee n)>(p\vee n)+1+N_{\psi,k}+k$) and $\epsilon_\delta=\epsilon_* \wedge\frac{\frac{1}{2}\delta\eta_1\eta_2}{\kappa_1(\kappa_2+\frac{1}{2}\delta\eta_1)}$. Then by using Proposition~\ref{mainprop} and~\eqref{uniupper},~\eqref{IdeltaI}, for any $\delta<\delta_K$, $\bar{q}\in K$, $\epsilon<\epsilon_\delta$ and $S\in\{G,R\}$ we have
\begin{eqnarray*}
&&\mathbb{E}\Bigg(\sup_{q\in U_\epsilon(\bar{q}) }  \mathcal{I}^{S}_{n,\delta}(q,\epsilon)\Bigg)\\
&\le& \sum_{p\ge 1} \Bigg(  \mathbb{E}\left(B_{p,m_p}\right)+\sum_{m\ge m_p}  \mathbb{E}\left(\Delta B_{p,m}\right)\Bigg) \\
&\le& \sum_{p\ge 1} c_1\cdot C_{p\vee n}\cdot 2^{c_2\cdot ((p\vee n) -p)} \cdot\Bigg( 2^n\cdot 2^{-\eta_1\delta \cdot p+\kappa_1\epsilon\cdot m_p} + \sum_{m\ge m_p} 2^{\kappa_2\cdot p-\eta_2\cdot m} \Bigg)\\
&\le & c_1\cdot \sum_{p\ge 1} C_{p \vee n}\cdot  2^{c_2\cdot ((p\vee n) -p)} \cdot\Bigg(  2^n\cdot 2^{-\eta_1\delta \cdot p+\kappa_1\cdot \epsilon_\delta\cdot m_p} + 2^{\kappa_2\cdot p-\eta_2\cdot m_p}\cdot\frac{1}{1-2^{-\eta_2}} \Bigg)\\
&\le& \frac{c_1}{1-2^{-\eta_2}}\cdot \sum_{p\ge 1} C_{p \vee n}\cdot 2^{c_2\cdot ((p\vee n) -p)} \cdot 2^{n} \cdot\\
&&\qquad \qquad \qquad \qquad \Bigg(2^{-\eta_1\delta \cdot p+\kappa_1\cdot \frac{\frac{1}{2}\delta\eta_1\eta_2}{\kappa_1(\kappa_2+\frac{1}{2}\delta\eta_1)}\cdot \frac{\kappa_2+\frac{1}{2}\delta\eta_1}{\eta_2}\cdot (n\vee p)}+2^{\kappa_2\cdot p-\eta_2\cdot \frac{\kappa_2+\frac{1}{2}\delta\eta_1}{\eta_2}\cdot p} \Bigg)\\
&=& \frac{2c_1}{1-2^{-\eta_2}}\cdot \sum_{p\ge 1} C_{p \vee n}\cdot  2^{(c_2+\frac{1}{2}\eta_1\delta)\cdot ((p\vee n) -p)} \cdot 2^{n} \cdot 2^{-\frac{1}{2}\delta \eta_1p} \\
&=& \frac{2^{n+1}c_1}{1-2^{-\eta_2}}\cdot \Bigg( \Big(\sup_{w\in\Sigma_{n+4+N_{\psi,k}}} \|f_w\|_\infty\Big) \cdot\sum_{p=1}^n 2^{(c_2+\frac{1}{2}\eta_1\delta)\cdot (n-p)} \cdot 2^{-\frac{1}{2}\delta \eta_1p} + \\
&& \qquad \qquad 2^{\frac{1}{2}\delta \eta_1(4+N_{\psi,k})}\cdot \sum_{p= n+1}^{\infty} \Big(\sup_{w\in\Sigma_{p+4+N_{\psi,k}}} \|f_w\|_\infty\Big) \cdot 2^{-\frac{1}{2}\delta \eta_1(p+4+N_{\psi,k})} \Bigg)<\infty,
\end{eqnarray*}
where the finiteness is ensured by assumption (A3). Since for any $0<\epsilon<\epsilon_\delta$, the family $\{U_\epsilon(\bar{q})\}_{\bar{q}\in K}$ forms an open covering of $K$, there exist $\bar{q}_1,\cdots,\bar{q}_N$ such that $\{U_\epsilon(\bar{q}_i)\}_{1\le i\le N}$ also covers $K$. This gives us the conclusion.

\end{proof}

\subsection{Proof of Proposition~\ref{mainprop}}
\begin{proof}
Due to~\eqref{Ulambda} we always have
$$\bigcup_{q\in U_\epsilon(\bar{q}) } E_n^{(k)}(q,\epsilon) \subset E_n^{(k)}(\bar{q},2\epsilon).$$
Then due to~\eqref{indicatorb} we have
\begin{eqnarray*}
&&\sup_{q\in U_\epsilon(\bar{q}) } \mu_q^{(k)}([u]\cap E_n^{(k)}(q,\epsilon)) \mu_q^{(k)}([v]\cap E_n^{(k)}(q,\epsilon))\\
&\le&\sup_{q\in U_\epsilon(\bar{q}) } \mathbf{1}_{\left\{[u]\cap E_n^{(k)}(q,\epsilon)\neq\emptyset\right\}} 2^{-|u|(D^{(k)}_q-\epsilon)} \cdot \mathbf{1}_{\left\{[v]\cap E_n^{(k)}(q,\epsilon)\neq\emptyset\right\}} 2^{-|v|(D^{(k)}_q-\epsilon)} \\
&\le& \mathbf{1}_{\left\{[u]\cap E_n^{(k)}(\bar{q},2\epsilon)\neq\emptyset\right\}} \cdot \mathbf{1}_{\left\{[v]\cap E_n^{(k)}(\bar{q},2\epsilon)\neq\emptyset\right\}} \cdot 2^{-(|u|+|v|)(D^{(k)}_{\bar{q}}-2\epsilon)}.
\end{eqnarray*}
This gives us 
\begin{multline}
B_{p,m} \le 2^{-2m(D^{(k)}_{\bar{q}}-2\epsilon)} \sum_{u,v\in \Sigma_m;\atop \mathbf{1}_p(u,v)=1} \mathcal{K}_{\gamma_{\bar{q},S}^{(k)}-\delta-\epsilon} (u,v) \cdot
\mathbf{1}_{\left\{[u]\cap E_n^{(k)}(\bar{q},2\epsilon)\neq\emptyset\right\}} \mathbf{1}_{\left\{[v]\cap E_n^{(k)}(\bar{q},2\epsilon)\neq\emptyset\right\}}, \label{S}
\end{multline}
\begin{multline}
\Delta B_{p,m} \le 2^{-2(m+1)(D^{(k)}_{\bar{q}}-2\epsilon)}\cdot \sum_{u,v\in \Sigma_m,u',v'\in\{0,1\}; \mathbf{1}_p(u,v)=1}  \\
\big| \mathcal{K}_{\gamma_{\bar{q},S}^{(k)}-\delta-\epsilon} (uu',vv')- \mathcal{K}_{\gamma_{\bar{q},S}^{(k)}-\delta-\epsilon} (u,v)\big|\cdot \mathbf{1}_{\left\{[uu']\cap E_n^{(k)}(\bar{q},2\epsilon)\neq\emptyset\right\}} \mathbf{1}_{\left\{[vv']\cap E_n^{(k)}(\bar{q},2\epsilon)\neq\emptyset\right\}}. \label{T}
\end{multline}

Now we deal with each term of the above sums individually.

Fix $p$ and $n$ in $\mathbb{N}^*$, let $r=p\vee n$, and fix $m \ge r+1+N_{\psi,k}+ k$. 

Fix a pair $u,v\in\Sigma_m$ with $\mathbf{1}_p(u,v)=1$, so $(u|_{p+1},v|_{p+1})\in \mathcal{P}_p$. 

Let
\begin{equation*}
\left\{
\begin{array}{l}
V:=\mathcal{K}_{\gamma_{\bar{q},S}^{(k)}-\delta-\epsilon} (u,v) \cdot
\mathbf{1}_{\left\{[u]\cap E_n^{(k)}(\bar{q},2\epsilon)\neq\emptyset\right\}} \mathbf{1}_{\left\{[v]\cap E_n^{(k)}(\bar{q},2\epsilon)\neq\emptyset\right\}}; \\
\Delta V:= {\scriptstyle \big| \mathcal{K}_{\gamma_{\bar{q},S}^{(k)}-\delta-\epsilon} (uu',vv')- \mathcal{K}_{\gamma_{\bar{q},S}^{(k)}-\delta-\epsilon} (u,v)\big|} \cdot \mathbf{1}_{\left\{[uu']\cap E_n^{(k)}(\bar{q},2\epsilon)\neq\emptyset\right\}} \mathbf{1}_{\left\{[vv']\cap E_n^{(k)}(\bar{q},2\epsilon)\neq\emptyset\right\}}.
\end{array}
\right.
\end{equation*}

Due to~\eqref{indicatora},~\eqref{indicatorb} and~\eqref{indicatorc}, if $[u]\cap E_n^{(k)}(\bar{q},2\epsilon)\neq\emptyset$, then for $l=r,\cdots,m$ we have
$$\mathbf{1}^{(a)}_{u|_l}(\bar{q},2\epsilon)\cdot \mathbf{1}^{(b)}_{u|_l}(\bar{q},2\epsilon)\cdot \mathbf{1}^{(c)}_{u|_l}(\bar{q},2\epsilon)=1.$$

Define
\begin{eqnarray}
\mathbf{1}_{u,v}^{\text{ran}}(\bar q,\epsilon) &=& \mathbf{1}^{(a)}_{u|_{l}}(\bar q,2\epsilon)\cdot  \mathbf{1}^{(c)}_{u|_r}(\bar q,2\epsilon)\cdot \mathbf{1}^{(c)}_{u}(\bar q,2\epsilon)\cdot \mathbf{1}^{(c)}_{v}(\bar q,2\epsilon); \label{ran}\\
\mathbf{1}_{u,v}^{\text{det}}(\bar q,\epsilon) &=& \mathbf{1}^{(b)}_{u|_r}(\bar q,2\epsilon)\cdot \mathbf{1}^{(b)}_{v|_r}(\bar q,2\epsilon)\cdot \mathbf{1}^{(b)}_{u}(\bar q,2\epsilon)\cdot \mathbf{1}^{(b)}_{v}(\bar q,2\epsilon), \label{det}
\end{eqnarray}
where ``$\text{ran}$" stands for random and ``$\text{det}$" stands for deterministic.

Since $[uu']\cap E_n^{(k)}(\bar q,2\epsilon)\neq\emptyset$ implies $[u]\cap E_n^{(k)}(\bar q,2\epsilon)\neq\emptyset$, we have
\begin{eqnarray*}
\mathbf{1}_{\left\{[uu']\cap E_n^{(k)}(\bar q,2\epsilon)\neq\emptyset\right\}} \cdot \mathbf{1}_{\left\{[vv']\cap E_n^{(k)}(\bar q,2\epsilon)\neq\emptyset\right\}} &\le& \mathbf{1}_{\left\{[u]\cap E_n^{(k)}(\bar q,2\epsilon)\neq\emptyset\right\}}\cdot\mathbf{1}_{\left\{[v]\cap E_n^{(k)}(\bar q,2\epsilon)\neq\emptyset\right\}} \\
&\le& \mathbf{1}_{u,v}^{\text{ran}}(\bar q,\epsilon)\cdot \mathbf{1}_{u,v}^{\text{det}}(\bar q,\epsilon).
\end{eqnarray*}
This implies
\begin{equation*}
V \le \bar{\mathcal{K}}_{\gamma_{\bar{q},S}^{(k)}-\delta-\epsilon} (u,v) \cdot \mathbf{1}^{\text{det}}_{u,v}(\bar q,\epsilon) \ \text{ and }\  \Delta V \le \Delta \bar{\mathcal{K}}_{\gamma_{\bar{q},S}^{(k)}-\delta-\epsilon} (u,v) \cdot \mathbf{1}^{\text{det}}_{u,v}(\bar q,\epsilon),
\end{equation*}
where
\begin{equation}\label{kdeltak}
\left\{
\begin{array}{l}
\bar{\mathcal{K}}_{\gamma_{\bar{q},S}^{(k)}-\delta-\epsilon} (u,v)=\mathcal{K}_{\gamma_{\bar{q},S}^{(k)}-\delta-\epsilon} (u,v)\cdot  \mathbf{1}^{\text{ran}}_{u,v}(\bar q,\epsilon); \\
\ \\
\Delta \bar{\mathcal{K}}_{\gamma_{\bar{q},S}^{(k)}-\delta-\epsilon} (u,v)=\big| \mathcal{K}_{\gamma_{\bar{q},S}^{(k)}-\delta-\epsilon} (uu',vv')- \mathcal{K}_{\gamma_{\bar{q},S}^{(k)}-\delta-\epsilon} (u,v)\big| \cdot  \mathbf{1}^{\text{ran}}_{u,v}(\bar q,\epsilon).
\end{array}
\right.
\end{equation}

Since $\mathbf{1}^{\text{det}}_{u,v}(\bar q,\epsilon)$ is deterministic, we have
\begin{eqnarray*}
\mathbb{E}(V) &\le& \mathbb{E}\Big( \bar{\mathcal{K}}_{\gamma_{\bar{q},S}^{(k)}-\delta-\epsilon}(u,v) \Big) \cdot \mathbf{1}^{\text{det}}_{u,v}(\bar q,\epsilon), \\
\mathbb{E}(\Delta V) &\le& \mathbb{E}\Big( \Delta \bar{\mathcal{K}}_{\gamma_{\bar{q},S}^{(k)}-\delta-\epsilon} (u,v) \Big) \cdot \mathbf{1}^{\text{det}}_{u,v}(\bar{q},\epsilon).
\end{eqnarray*}

Recall that in Remark~\ref{J12} we distinguished the cases $K\subset J_k'$ and $K\subset J_k''$ according to whether or not the corresponding power on the kernel is greater than $1$. Then, due to~\eqref{kernel}, once we have taken $\delta<\delta_K$ and $\epsilon<\epsilon_K$, only two situations are left:
\begin{equation*}
\mathcal{K}_{\gamma} (u,v)=\left\{
\begin{array}{ll}
\left(\big|F^{\text{pert}}_{\mu}(\lambda(u))-F^{\text{pert}}_{\mu}(\lambda(v))\big|^2+|\lambda(u)-\lambda(v)|^2\right)^{\frac{\gamma}{2}} \vee 1, & \text{if } \gamma>1,\\
\ \\
\big|F^{\text{pert}}_{\mu}(\lambda(u))-F^{\text{pert}}_{\mu}(\lambda(v))\big|^{\gamma} \vee 1, & \text{if } \gamma < 1,
\end{array}
\right.
\end{equation*}
where $\gamma=\gamma_{\bar{q},S}^{(k)}-\delta-\epsilon$. {\it Notice that when we take $K$ a compact subinterval of $J'_k$ or $J''_k$, $\gamma$ could never be equal to $1$. }

\medskip

Recall that $C_r=\sup_{w\in\Sigma_{r+4+N_{\psi,k}}} \|f_w\|_\infty$, where $f_w$ is bounded density function of $\pi_w$ given in (A3). We have the following two lemmas:

\begin{lemma}\label{kernelcondi}
There exists a constant $c_\gamma>0$ such that
\begin{eqnarray*}
\mathbb{E}\Big( \bar{\mathcal{K}}_{\gamma} (u,v) \Big) &\le& c_\gamma \cdot C_r\cdot \left\{
\begin{array}{ll}
2^{r\cdot (h^{(k)}_{\bar{q}}+2\epsilon)-p(1-\gamma)}, & \text{ if } \gamma>1, \\
2^{n} \cdot 2^{r\cdot (h^{(k)}_{\bar{q}}\gamma+4\epsilon)}, & \text{ if } \gamma<1,
\end{array}
\right.\\
\mathbb{E}\Big( \Delta \bar{\mathcal{K}}_{\gamma} (u,v) \Big) &\le& c_\gamma \cdot C_r \cdot  \left\{
\begin{array}{ll}
2^{p\cdot 3-m\cdot (h^{(k)}_{\bar{q}}-2\epsilon)}, & \text{ if } \gamma>1, \\
2^{r\cdot 3-m\cdot (h^{(k)}_{\bar{q}}-2\epsilon)}, & \text{ if } \gamma<1.
\end{array}
\right.
\end{eqnarray*}
\end{lemma}

\begin{lemma}\label{counting}
\[
\sum_{u,v\in \Sigma_m}  \mathbf{1}_p(u,v)\cdot \mathbf{1}^{\mathrm{det}}_{u,v}(\bar q,\epsilon) \le 3^{r-p+1} \cdot 2^{2m\cdot (D^{(k)}_{\bar q}+2\epsilon)-r\cdot (D^{(k)}_{\bar q}-2\epsilon)}.
\]
\end{lemma}

Now, due to Remark~\ref{J12}, we have the following three expression of $\gamma$:
\begin{equation*}
\left\{
\begin{array}{lll}
\gamma=D^{(k)}_{\bar q}+1-h^{(k)}_{\bar{q}} -\delta-\epsilon>1, & h^{(k)}_{\bar{q}}< D^{(k)}_{\bar q}, & \text{case (i)}, \\
\gamma=D^{(k)}_{\bar q}/h^{(k)}_{\bar{q}} -\delta-\epsilon <1, & h^{(k)}_{\bar{q}}> D^{(k)}_{\bar q}, & \text{case (ii)}, \\
\gamma=1-\delta-\epsilon <1, & h^{(k)}_{\bar{q}}< D^{(k)}_{\bar q}, & \text{case (iii)}.
\end{array}
\right.
\end{equation*}
Then, due to~\eqref{S},~\eqref{T}, Lemma~\ref{kernelcondi} and Lemma~\ref{counting}, we have
\begin{eqnarray*}
\mathbb{E}\left(B_{p,m}\right)&\le& c_\gamma \cdot C_r\cdot 3^{r-p+1} \cdot 2^{-2m(D^{(k)}_{\bar{q}}-2\epsilon)} \cdot 2^{2m(D^{(k)}_{\bar q}+2\epsilon)-r(D^{(k)}_{\bar q}-2\epsilon)} \\
&&\qquad\qquad\qquad \qquad\qquad\qquad\cdot \left\{
\begin{array}{ll}
2^{r(h^{(k)}_{\bar{q}}+2\epsilon)-p(h^{(k)}_{\bar{q}}-D^{(k)}_{\bar{q}}+\delta+\epsilon)}, & \text{(i)} \\
2^{n} \cdot 2^{r(D^{(k)}_{\bar{q}}-h^{(k)}_{\bar{q}}(\delta+\epsilon)+4\epsilon)}, & \text{(ii)}\\
2^{n} \cdot 2^{r(h^{(k)}_{\bar{q}}-h^{(k)}_{\bar{q}}(\delta+\epsilon)+4\epsilon)},  & \text{(iii)}
\end{array}
\right.\\
&=& c_\gamma \cdot C_r\cdot 3^{r-p+1} \cdot 2^{8m\epsilon}\cdot 2^{-(r-p)(D^{(k)}_{\bar q}-2\epsilon)} \cdot 2^{-p(D^{(k)}_{\bar q}-2\epsilon)}\\
&&\qquad\qquad \cdot \left\{
\begin{array}{ll}
2^{(r-p)(h^{(k)}_{\bar{q}}+2\epsilon)}\cdot 2^{p(D^{(k)}_{\bar{q}}-\delta+\epsilon)}, & \text{(i)} \\
2^{n} \cdot 2^{(r-p)(D^{(k)}_{\bar{q}}-h^{(k)}_{\bar{q}}(\delta+\epsilon)+4\epsilon)}\cdot 2^{p(D^{(k)}_{\bar{q}}-h^{(k)}_{\bar{q}}(\delta+\epsilon)+4\epsilon)}, & \text{(ii)}\\
2^{n} \cdot 2^{(r-p)(h^{(k)}_{\bar{q}}-h^{(k)}_{\bar{q}}(\delta+\epsilon)+4\epsilon)}\cdot 2^{p(h^{(k)}_{\bar{q}}-h^{(k)}_{\bar{q}}(\delta+\epsilon)+4\epsilon)},  & \text{(iii)}
\end{array}
\right.\\
&\le& c_\gamma \cdot C_r\cdot 3 \cdot 2^{(r-p)(\log_2 3+2)} \cdot 2^n\cdot  \left\{
\begin{array}{ll}
2^{-(\delta-3\epsilon)\cdot p+8\epsilon\cdot m}, & \text{(i)} \\
2^{-(h^{(k)}_{\bar{q}}(\delta+\epsilon)-6\epsilon)\cdot p+8\epsilon\cdot m}, & \text{(ii)}\\
2^{-(h^{(k)}_{\bar{q}}(\delta+\epsilon)-6\epsilon)\cdot p+8\epsilon\cdot m},  & \text{(iii)}
\end{array}
\right.\\
\end{eqnarray*}

The upper bound of $\mathbb{E}\left(\Delta B_{p,m}\right)$ is simpler, in all cases we have
\begin{eqnarray*}
\mathbb{E}\left(\Delta B_{p,m}\right) &\le& c_\gamma \cdot C_r\cdot 3^{r-p+1}\cdot 4 \cdot 2^{-2(m+1)(D^{(k)}_{\bar{q}}-2\epsilon)} \cdot \ 2^{2m(D^{(k)}_{\bar q}+2\epsilon)-r(D^{(k)}_{\bar q}-2\epsilon)}\\
&&\ \qquad\qquad\qquad\qquad\qquad\qquad \cdot 2^{(r-p)\cdot 3}\cdot 2^{p\cdot 3-m\cdot (h^{(k)}_{\bar{q}}-2\epsilon)} \\
&\le& c_\gamma \cdot C_r \cdot 3\cdot 2^{(r-p)(\log_23 +3)} \cdot 2^{p\cdot 3-m\cdot (h^{(k)}_{\bar{q}}-10\epsilon)}.
\end{eqnarray*}

{\it Notice that by construction we always have $h^{(k)}_{\bar q}\ge s_0-1/p_0>0$, then the existences of the parameters $c_1,c_2>0$, $\kappa_1,\kappa_2, \eta_1,\eta_2>0$ and $\epsilon_*>0$ are direct consequences of what we have obtained.}
\end{proof}

\subsection{Proof of Lemma~\ref{kernelcondi}.}\label{proofpropker}

\begin{proof}
Let $l=r+1+N_{\psi,k}$. Due to~\eqref{Fpert}, we have
\begin{eqnarray}
F^{\text{pert}}_{\mu}(\lambda(u))-F^{\text{pert}}_{\mu}(\lambda(v))&=&\sum_{w\in\Sigma_*} \pi_w\cdot d_w\cdot (\psi_w(\lambda(u))-\psi_w(\lambda(v))) \nonumber\\
&=& \pi_{u|_{l}} \cdot A+ B, \label{AB}
\end{eqnarray}
where
\begin{eqnarray}
A &=& d_{u|_{l}}\cdot (\psi_{u|_{l}}(\lambda(u))-\psi_{u|_{l}}(\lambda(v))); \label{defA}\\
B &=& \sum_{w\in\Sigma_*\setminus\{ u|_{l}\}} \pi_w\cdot d_w\cdot (\psi_w(\lambda(u))-\psi_w(\lambda(v))).\nonumber
\end{eqnarray}
By construction $A$ is deterministic, and $\pi_{u|_{l}}$ and $B$ are independent.

Since when $\mathbf{1}_u^{(b)}(\bar q,2\epsilon)=1$ we have $\mu_{\bar{q}}^{(k)}([u])\neq 0$, then~\eqref{gec} and the fact that $|u|=m\ge l+k$ yield 
\begin{equation}\label{lala}
\big|\psi_{u|_l}(\lambda(u))-\psi_{u|_{l}}(\lambda(v))\big| \ge \frac{c_{\psi,k}}{2}.
\end{equation}

For $u',v'\in \{0,1\}$ we can write
\begin{eqnarray}
&&F^{\text{pert}}_{\mu}(\lambda(uu'))-F^{\text{pert}}_{\mu}(\lambda(vv')) \nonumber \\
&=&\eta\cdot (F^{\text{pert}}_{\mu}(\lambda(u))-F^{\text{pert}}_{\mu}(\lambda(v)))+D=\eta\cdot (\pi_{u|_{l}}\cdot A+ B)+D\label{etaD}
\end{eqnarray}
where
\begin{eqnarray*}
\eta &=& \frac{\psi_{u|_{l}}(\lambda(uu'))-\psi_{u|_{l}}(\lambda(vv'))}{\psi_{u|_{l}}(\lambda(u))-\psi_{u|_{l}}(\lambda(v))};\\
D &=& \sum_{w\in\Sigma_*\setminus\{ u|_{l}\}} \pi_w\cdot d_w\cdot  \big(\psi_w(\lambda(uu'))-\psi_w(\lambda(vv'))-\eta\cdot (\psi_w(\lambda(u))-\psi_w(\lambda(v)))\big).
\end{eqnarray*}
We have that $\eta$ is deterministic, and $\pi_{u|_{l}}$ and $D$ are independent. Moreover, since $\psi$ is $r_0$-smooth, there exists a constant $C_\psi$ such that for any $x,y\in\mathbb{R}$ we have $|\psi(x)-\psi(y)|\le C_\psi|s-t|$. Due to~\eqref{lala}, this implies
\begin{eqnarray}
|\eta-1| &=& \left| \frac{\psi_{u|_{l}}(\lambda(uu'))-\psi_{u|_{l}}(\lambda(u))+\psi_{u|_{l}}(\lambda(v))-\psi_{u|_{l}}(\lambda(vv'))}{\psi_{u|_{l}}(\lambda(u))-\psi_{u|_{l}}(\lambda(v))}\right| \nonumber\\
&\le& \frac{2C_\psi}{c_{\psi,k}} \left(\Big|T_{u|_{l}}^{-1}(\lambda(uu'))-T_{u|_{l}}^{-1}(\lambda(u))\Big|+\Big|T_{u|_{l}}^{-1}(\lambda(v))-T_{u|_{l}}^{-1}(\lambda(vv'))\Big|\right) \nonumber\\
&\le & \frac{2C_\psi}{c_{\psi,k}}\cdot 2^{l}\cdot 2^{-m},\label{eta-1}
\end{eqnarray}
where we have used $|\lambda(u)-\lambda(uu')|\vee |\lambda(v)-\lambda(vv')| \le 2^{-m}$.

For $w\in\Sigma_*$ define the $\sigma$-algebra $\mathcal{A}_{w}=\sigma(\pi_u:u\in\Sigma_*\setminus\{w\})$.

By construction, $B$ and $D$ are $\mathcal{A}_{u|_{l}}$-measurable, thus are constant given $\mathcal{A}_{u|_{l}}$.

From assumption (A3) we know $\pi_{u|_{l}}$ has a bounded density function $f_{u|_{l}}$.

From~\eqref{indicatora},~\eqref{defA} and~\eqref{lala}, we have
\begin{equation}\label{A}
\mathbf{1}^{\text{ran}}_{u|_{l}}(\bar q,2\epsilon)\cdot |A|^{-1} \le \frac{2}{c_{\psi,k}}\cdot 2^{l(h^{(k)}_{\bar{q}}+2\epsilon)}.
\end{equation}

When $u,v\in \Sigma_m$ and $\mathbf{1}_p(u,v)=1$, we have
\begin{equation}\label{tuv}
\begin{cases}
|\lambda(u)-\lambda(v)-(\lambda(uu')-\lambda(vv'))|\le 2\cdot 2^{-m} \\
|\lambda(u)-\lambda(v)| \wedge |\lambda(uu')-\lambda(vv')|\ge 2^{-p-1}.
\end{cases}
\end{equation}

Since $\mathbf{1}_p(u,v)=1$ implies $v|_p\in\mathcal{N}(u|_p)$, by~\eqref{indicatorc} we have
\[
\mathbf{1}^{(c)}_{u|_p}(\bar q,2\epsilon)\cdot \sup_{s,t\in\bigcup_{w\in\mathcal{N}(u|_p)} [w]} |F^{\text{pert}}_{\mu}(\lambda(s))-F^{\text{pert}}_{\mu}(\lambda(t))| \le 2^{-p(h^{(k)}_{\bar{q}}-2\epsilon)}.
\]
This implies, when $\mathbf{1}^{(c)}_{u|_r}(\bar q,2\epsilon)=1$,
\begin{multline}
\Big( |F^{\text{pert}}_{\mu}(\lambda(u))-F^{\text{pert}}_{\mu}(\lambda(v))|\vee |F^{\text{pert}}_{\mu}(\lambda(uu'))-F^{\text{pert}}_{\mu}(\lambda(vv'))\big| \Big) \wedge 1 \\
\le 2^{-\mathbf{1}_{\{p\ge n\}}\cdot r(h^{(k)}_{\bar{q}}-2\epsilon)}:=\alpha.\label{alpha}
\end{multline}
Also, for the same reason, when $\mathbf{1}^{(c)}_{u}(\bar q,2\epsilon)\cdot \mathbf{1}^{(c)}_{v}(\bar q,2\epsilon)=1$,
\begin{multline}
\Big(|F^{\text{pert}}_{\mu}(\lambda(u))-F^{\text{pert}}_{\mu}(\lambda(uu'))\big|\vee \big|F^{\text{pert}}_{\mu}(t_{v})-F^{\text{pert}}_{\mu}(\lambda(vv'))\big| \Big)\\
\le 2^{-m(h^{(k)}_{\bar{q}}-2\epsilon)}:=\beta.\label{beta}
\end{multline}
These two inequalities with~\eqref{det},~\eqref{etaD} and ~\eqref{eta-1} imply that when $\mathbf{1}^{\text{ran}}_{u,v}(\bar q,\epsilon)=1$,
\begin{eqnarray}
|D|&\le& {\scriptstyle \big|F^{\text{pert}}_{\mu}(\lambda(uu'))-F^{\text{pert}}_{\mu}(t_{u})+F^{\text{pert}}_{\mu}(t_{v})-F^{\text{pert}}_{\mu}(\lambda(vv'))\big|+\big|\eta-1\big|\cdot \big|F^{\text{pert}}_{\mu}(t_{u})-F^{\text{pert}}_{\mu}(t_{v})\big|}\nonumber \\
&\le & 2\beta+ \frac{2C_\psi}{c_{\psi,k}}\cdot 2^{l}\cdot 2^{-m}\cdot \alpha\nonumber \\
&=& 2\cdot 2^{-m(h^{(k)}_{\bar{q}}-2\epsilon)}+  \frac{2C_\psi}{c_{\psi,k}}\cdot 2^{l}\cdot 2^{-m}\cdot 2^{-\mathbf{1}_{\{p\ge n\}}\cdot r(h^{(k)}_{\bar{q}}-2\epsilon)}\nonumber \\
&\le& 2\left(\frac{2C_\psi}{c_{\psi,k}}\cdot 2^{1+N_{\psi,k}}\right)\cdot 2^{-m(h^{(k)}_{\bar{q}}-2\epsilon)+r(1-\mathbf{1}_{\{p\ge n\}}\cdot (h^{(k)}_{\bar{q}}-2\epsilon))}\nonumber\\
&\le& C_D\cdot 2^{-m(h^{(k)}_{\bar{q}}-2\epsilon)+r},\label{D}
\end{eqnarray}
where $C_D=2\left(\frac{2C_\psi}{c_{\psi,k}}\cdot 2^{1+N_{\psi,k}}\right)$ and we have used $h^{(k)}_{\bar{q}}-2\epsilon\in(0,1)$.

\medskip

Recall that $l=r+1+N_{\psi,k}$. Now we have

\medskip

\begin{itemize}
\item[(I)] When $\gamma>1$, (also $\gamma\le 2$), due to~\eqref{kdeltak},~\eqref{kernel},~\eqref{A} and~\eqref{tuv},
\begin{eqnarray*}
&&\mathbb{E}\Big(\bar{\mathcal{K}}_\gamma(u,v) \Big| \mathcal{A}_{u|_{l}} \Big) \\
& \le & \int_{\mathbb{R}} \mathbf{1}_{\{ \bar{\mathcal{K}}_\gamma(u,v) = 1\}}\cdot f_{u|_{l}}(x)\ \mathrm{d}x + \int_{\mathbb{R}} \frac{ \mathbf{1}^{\text{ran}}_{u,v}(\bar q,\epsilon)\cdot  f_{u|_{l}}(x)}{(|A\cdot x+B|^2+|\lambda(u)-\lambda(v)|^2)^{\gamma/2}}\ \mathrm{d}x \\
&\le&1+ \int_{\mathbb{R}}  \mathbf{1}^{\text{ran}}_{u,v}(\bar q,\epsilon)\cdot |A|^{-1}|\lambda(u)-\lambda(v)|^{1-\gamma}\cdot \frac{ f_{u|_{l}}(\frac{|\lambda(u)-\lambda(v)|z-B}{A})}{(|z|^2+1)^{\gamma/2}}\ \mathrm{d} z\\
&\le& 1+ \frac{2}{c_{\psi,k}} \cdot 2^{l(h^{(k)}_{\bar{q}}+2\epsilon)} \cdot 2^{-(p+1)(1-\gamma)} \cdot \|f_{u|_l}\|_\infty \cdot \int_{\mathbb{R}} \frac{1}{(|z|^2+1)^{\gamma/2}}\ \mathrm{d} z \\
&= & 1+ \left(\int_{\mathbb{R}} \frac{\mathrm{d} z}{(|z|^2+1)^{\gamma/2}}\cdot\frac{2^{(1+N_{\psi,k})(h^{(k)}_{\bar{q}}+2\epsilon)+\gamma}}{c_{\psi,k}} \right)\cdot \|f_{u|_l}\|_\infty \cdot 2^{r(h^{(k)}_{\bar{q}}+2\epsilon)-p(1-\gamma)} \\
&\le& 2\left(\int_{\mathbb{R}} \frac{\mathrm{d} z}{(|z|^2+1)^{\gamma/2}}\cdot\frac{2^{(1+N_{\psi,k})(h^{(k)}_{\bar{q}}+2\epsilon)+\gamma}}{c_{\psi,k}} \right)\cdot C_r\cdot 2^{r(h^{(k)}_{\bar{q}}+2\epsilon)-p(1-\gamma)},
\end{eqnarray*}
where we recall that $C_r=\sup_{w\in\Sigma_{l}} \|f_w\|_\infty$;

\item[(II)] When $\gamma>1$, let
\[
\phi_\gamma(x,y)=(\big|F^{\text{pert}}_{\mu}(\lambda(u))-F^{\text{pert}}_{\mu}(\lambda(v))+x\big|^2+\big||\lambda(u)-\lambda(v)|+y\big|^2)^{-\gamma/2}.
\]
Then due to~\eqref{kdeltak},~\eqref{kernel},~\eqref{beta} and~\eqref{tuv}, we have
\begin{equation}
\Delta \bar{\mathcal{K}}_\gamma(u,v)
\le\int_{|y|\le 2\cdot 2^{-m}} \left|\frac{\partial }{\partial y}\phi_\gamma(0,y)\right|\ \mathrm{d} y + \int_{|x|\le \beta} \sup_{|y|\le 2\cdot 2^{-m}} \left|\frac{\partial }{\partial x}\phi_\gamma(x,y)\right|\ \mathrm{d}x, \label{con2.1}
\end{equation}
where we have used that $|a\vee 1-b\vee 1|\le |a-b|$ for any $a,b\ge 0$. It is not difficult to check that 
\[
\left|\frac{\partial }{\partial y}\phi_\gamma(0,y)\right|\vee \left|\frac{\partial }{\partial x}\phi_\gamma(x,y)\right| \le \gamma\cdot\big||\lambda(u)-\lambda(v)|+y\big|^{-\gamma-1}.
\]
In fact, we have
\begin{eqnarray*}
\left|\frac{\partial }{\partial y}\phi_\gamma(0,y)\right| 
&\le &\frac{\gamma\cdot \big||\lambda(u)-\lambda(v)|+y\big|}{\Big(\big|F^{\text{pert}}_{\mu}(\lambda(u))-F^{\text{pert}}_{\mu}(\lambda(v))\big|^2+\big||\lambda(u)-\lambda(v)|+y\big|^2\Big)^{1+\frac{\gamma}{2}}}\\
&\le &\frac{\gamma\cdot \big||\lambda(u)-\lambda(v)|+y\big|}{\big||\lambda(u)-\lambda(v)|+y\big|^{2+\gamma}}=\gamma\cdot \big||\lambda(u)-\lambda(v)|+y\big|^{-1-\gamma};
\end{eqnarray*}
and
\begin{eqnarray*}
&&\left|\frac{\partial }{\partial x}\phi_\gamma(x,y)\right| \\
&\le &\frac{\gamma\cdot \big|F^{\text{pert}}_{\mu}(\lambda(u))-F^{\text{pert}}_{\mu}(\lambda(v))+x\big|}{\Big(\big| F^{\text{pert}}_{\mu}(\lambda(u))-F^{\text{pert}}_{\mu}(\lambda(v))+x\big|^2+\big||\lambda(u)-\lambda(v)|+y\big|^2\Big)^{1+\frac{\gamma}{2}}}\\
&\le &\frac{\gamma\cdot \big|F^{\text{pert}}_{\mu}(\lambda(u))-F^{\text{pert}}_{\mu}(\lambda(v))+x\big| }{\big|F^{\text{pert}}_{\mu}(\lambda(u))-F^{\text{pert}}_{\mu}(\lambda(v))+x\big|^2+\big||\lambda(u)-\lambda(v)|+y\big|^2} \cdot \frac{1}{\big||\lambda(u)-\lambda(v)|+y\big|^{\gamma}}\\
&\le & \frac{\gamma}{2\big||\lambda(u)-\lambda(v)|+y\big|}\cdot \frac{1}{\big||\lambda(u)-\lambda(v)|+y\big|^{\gamma}} \le \gamma\cdot \big||\lambda(u)-\lambda(v)|+y\big|^{-1-\gamma},
\end{eqnarray*}
where we have used that $\frac{a}{a^2+b^2}\le \frac{1}{2b}$ for any $a,b>0$. This together with $|\lambda(u)-\lambda(v)|\ge 2^{-p-1}$, $m>p+1$ and $\gamma\le 2$ yields 
\begin{eqnarray*}
\Delta \bar{\mathcal{K}}_\gamma(u,v)  &\le& \gamma\cdot(2^{-p-1}-2\cdot 2^{-m})^{-\gamma-1}\cdot(2\cdot 2^{-m}+2\beta)\\
&\le& \gamma\cdot 2^{(p+2)(\gamma+1)}\cdot(2\cdot 2^{-m}+2 \cdot 2^{-m(h^{(k)}_{\bar{q}}-2\epsilon)}) \\
&\le& \gamma\cdot 2^{(p+2)(\gamma+1)}\cdot(4\cdot 2^{-m(h^{(k)}_{\bar{q}}-2\epsilon)}) \ \ \qquad (\text{since}\ h^{(k)}_{\bar{q}}-2\epsilon <1) \\
&\le& \left(4\gamma 2^{2(\gamma+1)}\right)\cdot C_r\cdot 2^{p\cdot 3-m\cdot (h^{(k)}_{\bar{q}}-2\epsilon)} .
\end{eqnarray*}

\item[(III)] When $\gamma<1$, due to~\eqref{kdeltak},~\eqref{kernel},~\eqref{A} and~\eqref{alpha},
\begin{eqnarray*}
&&\mathbb{E}\Big(\bar{\mathcal{K}}_{\gamma}(u,v) \Big| \mathcal{A}_{u|_{l}} \Big) \\
& \le & \int_{|A\cdot x+B|\le \alpha} \mathbf{1}_{\{\bar{\mathcal{K}}_{\gamma}(u,v)=1\}}\cdot f_{u|_{l}}(x)\ \mathrm{d}x+\int_{|A\cdot x+B|\le \alpha} \mathbf{1}^{\text{ran}}_{u,v}(\bar q,\epsilon)\cdot \frac{f_{u|_{l}}(x)}{|A\cdot x+B|^{\gamma}}\ \mathrm{d}x \\
&\le& 1+ \int_{|z|\le\alpha} \mathbf{1}^{\text{ran}}_{u,v}(\bar q,\epsilon) \cdot |A|^{-1} \cdot \frac{f_{u|_{l}}(\frac{z-B}{A})}{|z|^{\gamma}}\ \mathrm{d} z\\
&\le & 1+\frac{2}{c_{\psi,k}} \cdot 2^{l(h^{(k)}_{\bar{q}}+2\epsilon)} \cdot \|f_{u|_l}\|_\infty \cdot \int_{|z|\le\alpha} \frac{1}{|z|^{\gamma}}\ \mathrm{d} z \\
&=& 1+\frac{2}{c_{\psi,k}} \cdot 2^{l(h^{(k)}_{\bar{q}}+2\epsilon)} \cdot \|f_{u|_l}\|_\infty \cdot 2 \alpha^{1-\gamma}\\
&\le& 2\left(\frac{2^{(1+N_{\psi,k})(h^{(k)}_{\bar{q}}+2\epsilon)}}{c_{\psi,k}}\right)  \cdot \|f_{u|_l}\|_\infty\cdot  2^{r(h^{(k)}_{\bar{q}}+2\epsilon)-\mathbf{1}_{\{p\ge n\}}\cdot r(h^{(k)}_{\bar{q}}-2\epsilon)(1-\gamma)} \\
&=& 2\left(\frac{2^{(1+N_{\psi,k})(h^{(k)}_{\bar{q}}+2\epsilon)}}{c_{\psi,k}}\right) \cdot \|f_{u|_l}\|_\infty \cdot 2^{\mathbf{1}_{\{p< n\}}r(h^{(k)}_{\bar{q}}-2\epsilon)(1-\gamma)} \cdot 2^{r(h^{(k)}_{\bar{q}}\gamma+(4-2\gamma)\epsilon)} \\
&\le& 2\left(\frac{2^{(1+N_{\psi,k})(h^{(k)}_{\bar{q}}+2\epsilon)}}{c_{\psi,k}}\right) \cdot \|f_{u|_l}\|_\infty \cdot 2^{\mathbf{1}_{\{p< n\}}r} \cdot 2^{r(h^{(k)}_{\bar{q}}\gamma+4\epsilon)} \\
&\le& 2\left(\frac{2^{(1+N_{\psi,k})(h^{(k)}_{\bar{q}}+2\epsilon)}}{c_{\psi,k}}\right) \cdot C_r \cdot 2^{n} \cdot 2^{r(h^{(k)}_{\bar{q}}\gamma+4\epsilon)} ;
\end{eqnarray*}

\item[(IV)] When $\gamma<1$, due to~\eqref{kdeltak},~\eqref{kernel},~\eqref{A} and~\eqref{D}, by using again $|a\vee 1-b\vee 1|\le |a-b|$ for any $a,b\ge 0$, we have
\begin{eqnarray*}
&&\mathbb{E}\Big(\Delta \bar{\mathcal{K}}_\gamma(u,v)\Big| \mathcal{A}_{u|_{l}} \Big) \\
&\le& \int_{\mathbb{R}} \mathbf{1}^{\text{ran}}_{u,v}(\bar q,\epsilon)\cdot \Big|\frac{1}{|\eta(A\cdot x+B)+D|^\gamma}-\frac{1}{|A\cdot x+B|^\gamma}\Big|f_{u|_{l}}(x)\ \mathrm{d} x \\
&=&  \int_{\mathbb{R}} \mathbf{1}^{\text{ran}}_{u,v}(\bar q,\epsilon)\cdot |A|^{-1}\cdot |D|^{1-\gamma} \cdot \Big|\frac{1}{|\eta\cdot z+1|^\gamma}-\frac{1}{|z|^\gamma}\Big|f_{u|_{l}}(\frac{D\cdot z-B}{A})\ \mathrm{d} z \\
&\le&  \frac{2}{c_{\psi,k}} 2^{l(h^{(k)}_{\bar q}+2\epsilon)} \cdot (C_D\cdot 2^{-m(h^{(k)}_{\bar{q}}-2\epsilon)+r})^{1-\gamma} \cdot \|f_{u|_l}\|_\infty \cdot \int_{\mathbb{R}} \Big|\frac{1}{|\eta\cdot z+1|^\gamma}-\frac{1}{|z|^\gamma}\Big| \ \mathrm{d} z   \\
&=&\left(\frac{2^{1+(1+N_{\psi,k})(h^{(k)}_{\bar{q}}+2\epsilon)}C_D^{1-\gamma}}{c_{\psi,k}}  \cdot \int_{\mathbb{R}} \Big|\frac{1}{|\eta\cdot z+1|^\gamma}-\frac{1}{|z|^\gamma}\Big| \ \mathrm{d} z \right) \\
&&\qquad\qquad\qquad\qquad\qquad\qquad \cdot \|f_{u|_l}\|_\infty \cdot 2^{r(h^{(k)}_{\bar q}+2\epsilon +1-\gamma)-m(h^{(k)}_{\bar{q}}-2\epsilon)(1-\gamma)} \\
&\le&\left(\frac{2^{1+(1+N_{\psi,k})(h^{(k)}_{\bar{q}}+2\epsilon)}C_D^{1-\gamma}}{c_{\psi,k}} \cdot \int_{\mathbb{R}} \Big|\frac{1}{|\eta\cdot z+1|^\gamma}-\frac{1}{|z|^\gamma}\Big| \ \mathrm{d} z\right) \cdot C_r \cdot 2^{r\cdot 3-m(h^{(k)}_{\bar{q}}-2\epsilon)} .
\end{eqnarray*}

\end{itemize}

Now, since 
\[
\int_{\mathbb{R}} \frac{\mathrm{d} z}{(|z|^2+1)^{\gamma/2}}\ (\gamma>1) \text{ and } \int_{\mathbb{R}} \Big|\frac{1}{|\eta\cdot z+1|^\gamma}-\frac{1}{|z|^\gamma}\Big| \mathrm{d} z  \ (\gamma<1)
\]
are both finite (notice that $\eta$ is bounded away from $0$ and infinity uniformly), and $h^{(k)}_{\bar q}$ are chosen between $s_0-1/p_0$ and $1$, we can easily find a constant $c_\gamma$ such that
\begin{equation*}
\max \left( \left\{
\begin{array}{l}
2\left(\int_{\mathbb{R}} \frac{\mathrm{d} z}{(|z|^2+1)^{\gamma/2}}\cdot\frac{2^{(1+N_{\psi,k})(h^{(k)}_{\bar{q}}+2\epsilon)+\gamma}}{c_{\psi,k}} \right), 2\left(\frac{2^{(1+N_{\psi,k})(h^{(k)}_{\bar{q}}+2\epsilon)}}{c_{\psi,k}}\right),\\
\left(4\gamma 2^{2(\gamma+1)}\right), \left(\frac{2^{1+(1+N_{\psi,k})(h^{(k)}_{\bar{q}}+2\epsilon)}C_D^{1-\gamma}}{c_{\psi,k}} \cdot \int_{\mathbb{R}} \Big|\frac{1}{|\eta\cdot z+1|^\gamma}-\frac{1}{|z|^\gamma}\Big| \ \mathrm{d} z\right) 
\end{array}
\right.\right) \le c_\gamma.
\end{equation*}
This gives us the conclusion.

\end{proof}

\subsection{Proof of Lemma~\ref{counting}.}

\begin{proof}
Recall~\eqref{det} that
\[
\mathbf{1}^{\text{det}}_{u,v}(\bar q,\epsilon) = \mathbf{1}^{(b)}_{u|_r}(\bar q,2\epsilon)\cdot \mathbf{1}^{(b)}_{v|_r}(\bar q,2\epsilon)\cdot \mathbf{1}^{(b)}_{u}(\bar q,2\epsilon)\cdot \mathbf{1}^{(b)}_{v}(\bar q,2\epsilon).
\]
Let
\[
S_{p,m}=\sum_{u,v\in \Sigma_m} \mathbf{1}_p(u,v) \cdot \mathbf{1}^{(b)}_{u|_r}(\bar q,2\epsilon)\cdot \mathbf{1}^{(b)}_{v|_r}(\bar q,2\epsilon)\cdot \mathbf{1}^{(b)}_{u}(\bar q,2\epsilon)\cdot \mathbf{1}^{(b)}_{v}(\bar q,2\epsilon).
\]

Recall that $r=p\vee n$. For any $u\in\Sigma_m$ we write $u=u|_r\cdot u'$ with $u'\in\Sigma_{m-r}$. Since $ \mathbf{1}_p(u,v)$ only depends on $u|_r$, $v|_r$, we can write
\[
S_{p,m}=\sum_{u_r,v_r\in \Sigma_r} \mathbf{1}_p(u_r,v_r) \cdot \mathbf{1}^{(b)}_{u|_r}(\bar q,2\epsilon)\cdot \mathbf{1}^{(b)}_{v|_r}(\bar q,2\epsilon)\cdot \sum_{u',v'\in \Sigma_{m-r}}\mathbf{1}^{(b)}_{u|_r\cdot u'}(\bar q,2\epsilon)\cdot \mathbf{1}^{(b)}_{v|_r\cdot v'}(\bar q,2\epsilon).
\]

Recall (see~\eqref{indicatorb}) that
\[
\mathbf{1}^{(b)}_{u|_r\cdot u'}(\bar q,2\epsilon)=\mathbf{1}_{\left\{\mu_{\bar q}^{(k)}([u|_r\cdot u'])\in [2^{-m(D^{(k)}_{\bar q}+2\epsilon)}, 2^{-m(D^{(k)}_{\bar q}-2\epsilon)}]\right\}}.
\]
Thus
\[
\mathbf{1}^{(b)}_{u|_r\cdot u'}(\bar q,2\epsilon) \le 2^{m(D^{(k)}_{\bar q}+2\epsilon)} \cdot \mu_{\bar q}^{(k)}([u|_r\cdot u']).
\]
This implies that
\begin{eqnarray*}
&&\sum_{u',v'\in \Sigma_{m-r}}\mathbf{1}^{(b)}_{u|_r\cdot u'}(\bar q,2\epsilon)\cdot \mathbf{1}^{(b)}_{v|_r\cdot v'}(\bar q,2\epsilon) \\
&\le& 2^{2m(D^{(k)}_{\bar q}+2\epsilon)} \cdot \sum_{u',v'\in \Sigma_{m-r}} \mu_{\bar q}^{(k)}([u|_r\cdot u']) \cdot \mu_{\bar q}^{(k)}([v|_r\cdot v'])\\
&\le& 2^{2m(D^{(k)}_{\bar q}+2\epsilon)} \cdot \mu_{\bar q}^{(k)}([u|_r]) \cdot \mu_{\bar q}^{(k)}([v|_r]).
\end{eqnarray*}
Thus by the fact that given $u|_r$ in $\Sigma_r$, there are at most $3^{r-p+1}$ many $v|_r$ in $\Sigma_r$ such that $\mathbf{1}_{p}(u|_r,v|_r)=1$, we have
\begin{eqnarray*}
&&S_{p,m} \\
&\le& 2^{2m(D^{(k)}_{\bar q}+2\epsilon)}\cdot \sum_{u|_r,v|_r\in \Sigma_r} \mathbf{1}_p(u|_r,v|_r) \cdot \mathbf{1}^{(b)}_{u|_r}(\bar q,2\epsilon)\cdot \mathbf{1}^{(b)}_{v|_r}(\bar q,2\epsilon)\cdot \mu_{\bar q}^{(k)}([u|_r]) \cdot \mu_{\bar q}^{(k)}([v|_r])\\
&\le& 2^{2m(D^{(k)}_{\bar q}+2\epsilon)-r(D^{(k)}_{\bar q}-2\epsilon)}\cdot \sum_{u|_r,v|_r\in \Sigma_r} \mathbf{1}_p(u|_r,v|_r) \cdot \mu_{\bar q}^{(k)}([u|_r]) \\
&\le& 2^{2m(D^{(k)}_{\bar q}+2\epsilon)-r(D^{(k)}_{\bar q}-2\epsilon)}\cdot 3^{r-p+1}\cdot \sum_{u|_r\in \Sigma_r} \mu_{\bar q}^{(k)}([u|_r]) \\
&=& 2^{2m(D^{(k)}_{\bar q}+2\epsilon)-r(D^{(k)}_{\bar q}-2\epsilon)}\cdot 3^{r-p+1} \\
\end{eqnarray*}

\end{proof}

\medskip

\noindent
{\bf Acknowledgement.}

\medskip

\noindent
The author would like to thank gratefully his supervisor Professor Julien Barral for having suggested him to study the graph and range singularity spectra of random wavelet series, and his help in achieving this paper. He would also like to thank Doctor Yanhui Qu for some valuable discussions.

\end{document}